\documentclass[a4paper,10pt,reqno]{amsart}
\usepackage{marginnote,amsmath,amsfonts,amscd,amssymb,graphicx,mathrsfs,eufrak}
\usepackage[dvips,all,arc,curve,color,frame]{xy}
\usepackage[usenames]{color}

\usepackage[colorlinks]{hyperref}
\usepackage{tikz,mathrsfs}
\usetikzlibrary{arrows,decorations.pathmorphing,decorations.pathreplacing,positioning,shapes.geometric,shapes.misc,decorations.markings,decorations.fractals,calc,patterns}

\tikzset{>=stealth',
         cvertex/.style={circle,draw=black,inner sep=1pt,outer sep=3pt},
         vertex/.style={circle,fill=black,inner sep=1pt,outer sep=3pt},
         star/.style={circle,fill=yellow,inner sep=0.75pt,outer sep=0.75pt},
         tvertex/.style={inner sep=1pt,font=\scriptsize},
         gap/.style={inner sep=0.5pt,fill=white}}

\newcommand{\arrow}[2][20]
 {
  \hspace{-5pt}
  \begin{tikzpicture}
   \node (A) at (0,0) {};
   \node (B) at (#1pt,0) {};
   \draw [#2] (A) -- (B);
  \end{tikzpicture}
  \hspace{-5pt}
 }
\newcommand{\arrowl}[3][20]
 {
  \hspace{-5pt}
  \begin{tikzpicture}
   \node (A) at (0,0) {};
   \node (B) at (#1pt,0) {};
   \draw [#2] (A) -- node [above] {$#3$} (B);
  \end{tikzpicture}
  \hspace{-5pt}
 }

\newtheorem{theorem}{Theorem}[section]

\newtheorem{lemma}[theorem]{Lemma}
\newtheorem{definition}[theorem]{Definition}
\newtheorem{corollary}[theorem]{Corollary}

\theoremstyle{definition} 
 
\newtheorem{example}[theorem]{Example}

\newtheorem{remark}[theorem]{Remark}

\newcommand{\dbcoh}[1]{D^b(\coh{#1})}
\newcommand{\dbqcoh}[1]{D^b(\Qcoh{#1})}
\newcommand{\dcoh}[1]{D(\coh{#1})}
\newcommand{\da}[1]{D(\Mod{#1})}
\newcommand{\ka}[1]{K(\Mod{#1})}
\newcommand{\kqcoh}[1]{K(\Qcoh{#1})}
\newcommand{\ca}[1]{C(\Mod{#1})}
\newcommand{\dba}[1]{D^b(\Mod{#1})}
\newcommand{\db}[1]{D^b(\fgmod{#1})}
\newcommand{\dplusa}[1]{D^+(\Mod{#1})}
\newcommand{\dplus}[1]{D^+(\fgmod{#1})}
\newcommand{\dminusa}[1]{D^-(\Mod{#1})}
\newcommand{\dminus}[1]{D^-(\fgmod{#1})}
\newcommand{\kminuspa}[1]{K^-(\Proj{#1})}
\newcommand{\kminusa}[1]{K^-(\Mod{#1})}
\newcommand{\kbpa}[1]{K^b(\Proj{#1})}
\newcommand{\kbp}[1]{K^b(\fgProj{#1})}
\newcommand{\dqcoh}[1]{D(\Qcoh{#1})}

\newcommand{\Per}{{}^{-1}\mathfrak{Per}\w{X}}

\newcommand{\C}[1]{\mathbb{C}^{#1}}
\newcommand{\m}{\mathfrak{m}}
\newcommand{\sub}[1]{\langle \c{#1}\rangle}
\newcommand{\p}{\mathfrak{p}}
\newcommand{\pinspec}[1]{\p\in\textnormal{Spec}#1}
\newcommand{\s}[1]{\mathscr{#1}}
\renewcommand{\c}[1]{\mathcal{#1}}
\renewcommand{\u}[1]{\underline{#1}}
\newcommand{\uc}[1]{\u{\c{#1}}}
\renewcommand{\P}{\mathbb{P}^1}
\newcommand{\f}[1]{\mathfrak{#1}}
\newcommand{\ve}{\varepsilon}
\newcommand{\pull}[2]{\pi^{*}_{#1}\O{#2}}
\newcommand{\pullb}[1]{\pi^{*}_{#1}\mathscr{O}}

\newcommand{\w}[1]{\widetilde{#1}}
\newcommand{\e}{\varepsilon}
\renewcommand{\t}[1]{\textnormal{#1}}
\renewcommand{\tt}[1]{\mathtt{#1}}

\newcommand{\ver}[1]{\xy
\POS (0.8,0) *\cir<2pt>{}+{\txt{#1}} \endxy}
\def\minus{\hbox{-}}
\def\plus{\hbox{+}}
\newcommand{\sq}[1]{\xy
\POS(0,0)*+\hbox{#1} *\frm{-,}\endxy}
\newcommand{\verr}[1]{\xy
\POS (2,0) *\cir<4pt>{}+{\txt{#1}} \endxy}

\def\Cl{\mathop{\rm Cl}\nolimits}
\def\op{\mathop{\rm op}\nolimits}
\def\GL{\mathop{\rm GL}\nolimits}
\def\SL{\mathop{\rm SL}\nolimits}
\def\CM{\mathop{\rm CM}\nolimits}
\def\SCM{\mathop{\rm SCM}\nolimits}
\def\OCM{\mathop{\Omega{\rm CM}}\nolimits}
\def\depth{\mathop{\rm depth}\nolimits}
\def\fl{\mathop{\rm fl}\nolimits}
\def\hgt{\mathop{\rm ht}\nolimits}
\def\mod{\mathop{\rm mod}\nolimits}
\def\coh{\mathop{\rm coh}\nolimits}
\def\Mod{\mathop{\rm Mod}\nolimits}
\def\refl{\mathop{\rm ref}\nolimits}
\def\proj{\mathop{\rm proj}\nolimits}
\def\ind{\mathop{\rm ind}\nolimits}
\def\Kdim{\mathop{\rm Kdim}\nolimits}
\def\pd{\mathop{\rm proj.dim}\nolimits}
\def\id{\mathop{\rm id}\nolimits}
\def\Hom{\mathop{\rm Hom}\nolimits}
\def\hom{\mathop{\rm hom}\nolimits}
\def\End{\mathop{\rm End}\nolimits}
\def\Ext{\mathop{\rm Ext}\nolimits}
\def\ext{\mathop{\rm ext}\nolimits}
\def\H{\mathop{\rm H^{}}\nolimits}
\def\h{\mathop{\rm h^{}}\nolimits}
\def\Tor{\mathop{\rm Tor}\nolimits}
\def\Tr{\mathop{\rm Tr}\nolimits}
\def\add{\mathop{\rm add}\nolimits}
\def\Cok{\mathop{\rm Cok}\nolimits}
\def\Ker{\mathop{\rm Ker}\nolimits}
\def\Im{\mathop{\rm Im}\nolimits}
\def\Sing{\mathop{\rm Sing}\nolimits}
\def\Supp{\mathop{\rm Supp}\nolimits}
\def\Ass{\mathop{\rm Ass}\nolimits}
\def\Spec{\mathop{\rm Spec}\nolimits}
\def\Max{\mathop{\rm Max}\nolimits}
\def\gl{\mathop{\rm gl.dim}\nolimits}
\def\AA{\mathop{\mathcal{A}}\nolimits}
\def\CC{\mathop{\mathcal{C}}\nolimits}
\def\DD{\mathop{\mathcal{D}}\nolimits}
\def\FF{\mathop{\mathcal{F}}\nolimits}
\def\XX{\mathop{\mathcal{X}}\nolimits}
\def\D{\mathop{\rm{D}^{}}\nolimits}
\def\Db{\mathop{\rm{D}^b}\nolimits}
\def\Kb{\mathop{\rm{K}^b}\nolimits}

\begin{document}

\title{\textsc The $\GL(2,\C{})$ McKay Correspondence}
\author{Michael Wemyss}
\address{Mathematical Institute, 24-29 St Giles', Oxford, OX1 3LB, UK.} \email{wemyss.m@googlemail.com}
\maketitle

\begin{abstract}
In this paper we show that for any affine complete rational surface singularity the quiver of the reconstruction algebra can be determined combinatorially from the dual graph of the minimal resolution.    As a consequence the derived category of the minimal resolution is equivalent to the derived category of an algebra whose quiver is determined by the dual graph.  Also, for any finite subgroup $G$ of $\GL(2,\C{})$, it means that the endomorphism ring of the special CM $\C{}[[x,y]]^G$-modules can be used to build the dual graph of the minimal resolution of $\C{2}/G$, extending McKay's observation for finite subgroups of $\SL(2,\C{})$ to all finite subgroups of $\GL(2,\C{})$.
\end{abstract}

\parindent 20pt
\parskip 0pt
\section{Introduction}\label{Section:Intro}
When working with quotient singularities $V/G$ from the viewpoint of resolution of singularities and derived categories the object one first writes down is the skew group ring $\mathbb{C}[V]\# G$ (=$G$-equivariant sheaves) since this object is the canonical way to encode the representation theory of $G$ into the geometry.  This algebra satisfies many nice homological properties, in particular it has finite global dimension and so we often view it as a `non-commuatitive resolution' over its centre.  For quotients  with $G\leq \SL(V)$ there is much evidence which suggests that this is a good idea, but there is always a problem extracting and understanding the geometric information that $\C{}[V]\# G$ encodes.

The point in this paper is that for non-Gorenstein surface quotient singularities (i.e. those by finite small groups $G\leq\GL(2,\C{})$ that are not inside $\SL(2,\C{})$), in order to be able to see the geometry in the most clear way the skew group ring $\C{}[x,y]\# G$ is far too large; one should instead pass to a much smaller algebra, the so-called reconstruction algebra.  The benefit of passing to this smaller algebra is twofold --- firstly we are able to recover the link with the dual graph of the minimal resolution which is obscured in the world of $G$-equivariant sheaves (which we do in this paper), and secondly it is much easier to extract the geometry (which we do in other papers) since via the reconstruction algebra understanding resolutions of quotients by large non-abelian groups can turn out to be as easy as understanding those arising from vastly smaller cyclic groups.

For non-Gorenstein surface quotient singularities the number of exceptional curves in the minimal resolution of $\C{2}/ G$ is strictly less than the number of irreducible representations. This problem led Wunram \cite{Wunram_generalpaper} to develop the
idea of a special representation (equivalently special CM module) so that after passing to the
non-trivial indecomposable special representations the 1-1 correspondence with the
exceptional curves is recovered. Note that the definition of special
representation is homological since it is defined by the
vanishing of cohomology of the dual of a certain vector bundle on the minimal
resolution.  Even in fairly easy examples, determining which representations are special is not a straightforward task \cite{Iyama_Wemyss_specials}. 

Wunram's results give the necessary 1-1 correspondence and so here we study  the non-commutative algebra given by the endomorphism ring of the sum of the  indecomposable special CM modules.  In \cite{Wemyss_reconstruct_A} it was discovered algebraically that for finite small cyclic subgroups of $\GL(2,\C{})$ the quiver of this non-commutative algebra determines and is determined by the dual graph of  the minimal resolution $\hat{X}$ of the singularity $\mathbb{C}^2/G$, labelled with self-intersection numbers.  This correspondence is purely on the level of the underlying quiver; it was further discovered that if we add in the extra information of the relations then in fact one can recover the whole space $\hat{X}$ (not just the dual graph) as a certain GIT quotient, and also that the endomorphism algebra describes the derived category of $\hat{X}$.   Furthermore the global dimension of the endomorphism algebra was found to be either 2 or 3, which since we are studying surfaces is a little surprising.

In this paper we relate Wunram's work to that of Bridgeland \cite{Bridgeland} and  Van den Bergh \cite{VdB} to give a non-explicit geometric proof of some of the above results which furthermore works in greater generality.  Note that by using Riemann-Roch and Serre duality the proofs are quite routine, giving not only the number of arrows but also the number of relations.

Our first main theorem is the following:

\begin{theorem}[=\ref{main} and \ref{up=down}]\label{intromain}
Let $R$ be any affine complete rational surface singularity and let $\hat{X}$ be the minimal resolution of $\Spec R$.  Then the quiver of $\End_{R}( \bigoplus M )$, where the sum is taken over all indecomposable special CM $R$-modules, can be computed combinatorially from the dual graph of $\hat{X}$ labelled with self-intersection numbers.
\end{theorem}

We explain the terminology in \S\ref{Section:Prelim} and state this more precisely in \S\ref{Section:Exts} however we note here that the combinatorics are easy and so the computation of the quiver is very quick.  We call the endomorphism ring in the above correspondence the \emph{reconstruction algebra}, since its quiver can be reconstructed from the dual graph.  Of course, as is the case for preprojective algebras in the classical McKay Correspondence, to obtain $\End_{R}( \bigoplus M )$ from the dual graph requires us to add information in the form of an extra vertex, whereas to obtain the dual graph from $\End_{R}( \bigoplus M )$ one must lose information by killing a specified vertex.  

Using the computation of the ext groups in the proof of the above theorem we also obtain the following, extending the result \cite[2.15]{Iyama_Wemyss_specials}:

\begin{corollary}[=\ref{gldim2}]\label{introgldim}
Let $R$ be an affine complete rational surface singularity and set $A:=\End_{R}(\bigoplus M)$ where the sum is taken over all indecomposable special CM $R$-modules.  Then 
\[
\gl A=\left\{\begin{array}{cl} 2& \mbox{if $R$ is Gorenstein}\\ 3 & \mbox{else.}\end{array}\right. 
\]
When $R$ is Gorenstein all simple left $A$-modules and all simple right $A$-modules have projective dimension 2.  When $R$ is not Gorenstein all simple right $A$-modules have projective dimension 2 except the simple corresponding to $\star$, which has projective dimension 3.  As left $A$-modules, the projective dimension of the simples at $\star$ and all curves corresponding to $(-2)$-curves have projective dimension 2, whereas all other simples have projective dimension 3.
\end{corollary}
Thus not only does the homologically homogeneous property fail for reconstruction algebras, it fails asymmetrically.  

Applying Theorem~\ref{intromain} to quotients of $\C{2}$ by finite subgroups of $\GL(2,\C{})$ we also obtain the result which motivated this work.
\begin{corollary}[The $\GL(2,\C{})$ McKay Correspondence]\label{GLintro}
For any finite small subgroup $G$ of $\GL(2,\C{})$ let $\hat{X}\rightarrow\C{2}/G$ be the minimal resolution. Then the special representations of $G$ can be used to build the dual graph of $\hat{X}$ by taking the quiver of $\End_{\C{}[[x,y]]^G}( \oplus_{\rho\,\,\t{special}} (\rho\otimes \C{}[[x,y]])^G)$ and deleting the vertex corresponding to the trivial representation.  
\end{corollary}

This says, provided we use Auslander's endomorphism ring perspective \cite{Auslander_AR=McKay}, that the representation theory of the special CM modules determines the dual graph of the minimal resolution in exactly the same way as in the classical $\SL(2,\C{})$ case.  Again we emphasize that the reconstruction algebra is in general much smaller than $G$-equivariant sheaves, allowing us to extract the geometry much more easily.

Although the geometric proof of Theorem~\ref{intromain} is quite slick, the main content in reconstruction algebras is found in their relations and other than telling us their number the geometric proof does nothing more. Furthermore $\End_{\C{}[[x,y]]^G}( \oplus_{\rho\,\,\t{special}} (\rho\otimes \C{}[[x,y]])^G)$ is an entirely representation-theoretic gadget and so turning Corollary~\ref{GLintro} around we should be able to deduce the dual graph by using only the representation theory.

This motivates us to provide a second proof of Corollary~\ref{GLintro} in which we show how to compute $\End_{\C{}[[x,y]]^G}( \oplus_{\rho\,\,\t{special}} (\rho\otimes \C{}[[x,y]])^G)$  purely representation-theoretically, assuming none of the geometry.  This method actually turns out to give more information since it immediately tells us the grading of the algebra and it also provides a method to obtain the relations.  For brevity in this paper we restrict ourselves to only determining the quiver and so we ignore the relations; we consider this problem in separate papers \cite{Wemyss_reconstruct_A,Wemyss_reconstruct_D(i),Wemyss_reconstruct_D(ii)} by using a combination of the two approaches.

The main idea behind the representation-theoretic proof for quotient singularities is that the number of irreducible maps between the special CM modules can be  determined by using a simple counting argument on the AR quiver of $\C{}[[x,y]]^{G}$.  The relations on the reconstruction algebra are then induced by the  mesh relations on the AR quiver.  This method involves case-by-case analysis and relies on the classification of the special CM modules in \cite{Iyama_Wemyss_specials} and so is quite space-intensive, thus we prove only   a few examples. 

The structure of this paper is as follows: we begin in \S\ref{Section:Prelim}
with preliminaries involving intersection theory and perverse sheaves.  In \S\ref{Section:Exts} we give the geometric proof of the main results above, and in \S\ref{Section:RepMethod} we give the alternative representation-theoretic method for quotient singularities.  The remainder of the paper is devoted to the quotient case: in \S\ref{groups} we translate the intersection theory into simple rules and in the remaining sections we draw the quivers for all finite subgroups of $\GL(2,\C{})$.

\section{Preliminaries}\label{Section:Prelim}

Throughout this paper let $X=\Spec R$ be an affine complete rational surface singularity over $\C{}$, let $f:\w{X}\rightarrow \Spec R$ be some resolution and denote the exceptional curves by $\{ E_i\}_{i\in I}$.  Resolutions will not be minimal unless specified.  We shall always assume that our dual graphs are labelled with the corresponding self-intersection numbers:
\begin{definition}
Suppose $\{ E_i \}_{i\in I}$  is a collection of $\mathbb{P}^1$s forming the exceptional locus in a resolution of some affine rational surface singularity.  The dual graph is defined as follows: for each curve draw a vertex, and join two vertices if and only if the corresponding curves intersect.  Furthermore label every vertex with the self-intersection number of the corresponding curve.
\end{definition}
\begin{definition}[\cite{Art66}]
For a given exceptional $\{ E_i \}_{i\in I}$, define the fundamental
cycle $Z_f=\sum_{i\in I}r_i E_i$ (with each $r_i\geq 1$) to
be the unique smallest element such that $Z_f\cdot E_i\leq 0$ for all 
$i\in I$.
\end{definition}
Our notational convention for writing $Z_f$ is as follows:
\begin{example}\label{fundamental_cycle_examples}
\[
\begin{tikzpicture}[yscale=1.2]
\node at (-3,0) {(i)};
\node at (0,0) {\begin{tikzpicture}[xscale=0.85,yscale=1]
 \node (0) at (0,0) [vertex] {};
 \node (1) at (1,0) [vertex] {};
 \node (1b) at (1,0.75) [vertex] {};
 \node (2) at (2,0) [vertex] {};
 \node (3) at (3,0) [vertex] {};
 \node (4) at (4,0) [vertex] {};
 \node (5) at (5,0)[vertex] {};
 \node (0a) at (-0.1,-0.25) {$\scriptstyle - 2$};
 \node (1a) at (0.9,-0.25) {$\scriptstyle -2$};
 \node (1ba) at (0.65,0.75) {$\scriptstyle - 2$};
 \node (2a) at (1.9,-0.25) {$\scriptstyle - 2$};
 \node (2a) at (2.9,-0.25) {$\scriptstyle - 5$};
 \node (4a) at (3.9,-0.25) {$\scriptstyle -2$};
 \node (5a) at (4.9,-0.25) {$\scriptstyle - 3$};
\draw [-] (0) -- (1);
\draw [-] (1) -- (2);
\draw [-] (2) -- (3);
\draw [-] (3) -- (4);
\draw [-] (4) -- (5);
\draw [-] (1) -- (1b);
\end{tikzpicture}};
\node at (3.5,0) {$\scriptstyle Z_{f}=$};
\node at (6,0.1)
{\begin{tikzpicture}[xscale=0.75,yscale=1]
 \node (0) at (0,0)  {$\scriptstyle 1$};
 \node (1) at (1,0)  {$\scriptstyle 2$};
 \node (1b) at (1,0.75)  {$\scriptstyle 1$};
 \node (2) at (2,0)  {$\scriptstyle 2$};
 \node (3) at (3,0) {$\scriptstyle 1$};
 \node (4) at (4,0)  {$\scriptstyle 1$};
 \node (5) at (5,0) {$\scriptstyle 1$};
\end{tikzpicture}};
\node at (-3,-1.5) {(ii)};
\node at (0,-1.5) {\begin{tikzpicture}[xscale=0.85,yscale=1]
 \node (0) at (0,0) [vertex] {};
 \node (1) at (1,0) [vertex] {};
 \node (1b) at (2,0.75) [vertex] {};
 \node (2) at (2,0) [vertex] {};
 \node (3) at (3,0) [vertex] {};
 \node (4) at (4,0) [vertex] {};
 \node (0a) at (-0.1,-0.25) {$\scriptstyle - 2$};
 \node (1a) at (0.9,-0.25) {$\scriptstyle -2$};
 \node (1ba) at (1.65,0.75) {$\scriptstyle - 2$};
 \node (2a) at (1.9,-0.25) {$\scriptstyle - 2$};
 \node (2a) at (2.9,-0.25) {$\scriptstyle - 3$};
 \node (4a) at (3.9,-0.25) {$\scriptstyle -2$};
\draw [-] (0) -- (1);
\draw [-] (1) -- (2);
\draw [-] (2) -- (3);
\draw [-] (3) -- (4);
\draw [-] (2) -- (1b);
\end{tikzpicture}};
\node at (3.5,-1.5) {$\scriptstyle Z_{f}=$};
\node at (6,-1.4)
{\begin{tikzpicture}[xscale=0.75,yscale=1]
 \node (0) at (0,0)  {$\scriptstyle 1$};
 \node (1) at (1,0)  {$\scriptstyle 2$};
 \node (1b) at (2,0.75)  {$\scriptstyle 1$};
 \node (2) at (2,0)  {$\scriptstyle 2$};
 \node (3) at (3,0) {$\scriptstyle 1$};
 \node (4) at (4,0)  {$\scriptstyle 1$};
\end{tikzpicture}};
\end{tikzpicture}
\]\end{example}

As is standard, we denote the canonical cycle by $Z_K$.  It is the rational cycle defined by the condition $Z_K\cdot E_i=-K_{\w{X}}\cdot E_i$ for all $i\in I$.  By adjunction this means that $Z_K\cdot E_i=E_i^2+2$ for all $i\in I$.  Note that if the resolution is minimal then $Z_K\cdot E_i\leq 0$ for all $i\in I$.

Now perverse sheaves were introduced by Bridgeland \cite{Bridgeland} to prove the existence of flops of certain 3-folds; here we use this theory for surfaces.  The key point from our perspective is the following commutative diagram, proved by Van den Bergh  \cite[3.2.8, 3.5.5]{VdB}
\[
\xymatrix@R=3pt@C=20pt{{\Db(\coh\w{X}})\ar[r]^(0.32){\approx}& \Db(\mod\End_{\w{X}}(\c{O}_{\w{X}}\oplus \c{M}_{I}))\\ \cup & \cup\\ {\Per}\ar[r]^(0.35){\approx}&\mod\End_{\w{X}}(\c{O}_{\w{X}}\oplus \c{M}_{I})}
\]
where $\c{M}_{I}=\bigoplus_{i\in I} \c{M}_{i}$ with each $\c{M}_{i}$ a certain vector bundle satisfying $\det\c{M}_{i}\cdot E_{j}=\delta_{ij}$.  Furthermore the simple modules in $\mod \End_{\w{X}}(\c{O}_{\w{X}}\oplus \c{M}_{I})$ are, viewed inside $\Db(\coh \w{X})$, precisely \cite[3.5.7]{VdB}
\[
\c{O}_{Z_f} \quad\t{and}\quad \c{O}_{E_i}(-1)[1] \t{ for all } i\in I.
\]
The same set of objects also appears in the work of Ishii \cite{Ishii} for quotient singularities as the homology of the lift of the AR sequences to the minimal resolution.  Since we are assuming $R$ is complete, for each exceptional curve $E_i$ there is a divisor $D_i$ intersecting $E_i$ transversally at one point and not intersecting any of the other exceptional curves (see \cite[3.4.4]{VdB}).  Thus defining $D:=\sum_{i\in I}D_i$ then the simples are $\c{O}_{Z_f}$ and $\c{O}_{E_i}(-D)[1]$ for all $i\in I$.

\section{Ext Groups and Corollaries}\label{Section:Exts}

When working with quivers by $ab$ we mean $a$ followed by $b$; similarly when composing morphisms by $fg$ we mean $f$ followed by $g$.  With these conventions representations of quivers correspond to right modules, and further the functor $\t{RHom}(\c{O}_{\w{X}}\oplus \c{M}_{I},-)$ takes us to the derived category of left $\End_{\w{X}}(\c{O}_{\w{X}}\oplus \c{M}_{I})$ modules.  In what follows denote the simples in $\Per$ (i.e. left $\End_{\w{X}}(\c{O}_{\w{X}}\oplus \c{M}_{I})$-modules) by $S_\star=\c{O}_{Z_f}$, and $S_i=\c{O}_{E_i}(-D)[1]$ for all $i\in I$.

The following is immediate from the definition of $Z_f$:
\begin{lemma}\label{inequality}
If $E_i$ is some exceptional curve then $E_i^2\leq Z_f\cdot E_i$.  Furthermore the inequality is strict if $E_i$ intersects some other exceptional curve.
\end{lemma}
\begin{proof}
Write $Z_f=\sum a_iE_i$.  If $a_i=1$ then 
\[
Z_f\cdot E_i=E_i^2+\sum_{i\neq j}a_jE_j\cdot E_i\geq E_i^2
\]
where the inequality is strict provided that $E_i\cdot E_j=1$ for some $j$.  If $a_i>1$ write $Z^\prime =Z_f-E_i=(a_i-1)E_i+\sum_{i\neq j}a_jE_j$ then $Z^\prime$ has all co-efficients $\geq 1$ with $Z^\prime\cdot E_j=(Z_f-E_i)\cdot E_j\leq 0$ for all $j\neq i$.  Further $Z^\prime\cdot E_i=Z_f\cdot E_i-E_i^2$, thus $Z_f\cdot E_i>E_i^2$ else $Z_f$ is not minimal. 
\end{proof}

Let $e$ denote the embedding dimension of $R$, i.e. $e-2=-Z_K\cdot Z_f+1=-1-Z_f\cdot Z_f$.  In what follows, for every integer $a\in\mathbb{Z}$ denote
\[
 a_+:=\left\{\begin{array}{ccc}a& \t{if}& a\geq 0\\ 0 & \t{if}&a<0\end{array}\right. \quad \t{and}\quad a_-=\left\{\begin{array}{rcc}0& \t{if} &a\geq 0\\ -a & \t{if}&a<0\end{array}\right. .
 \]

\begin{theorem}\label{extcalculation}
Let $\w{X}\rightarrow \Spec R$ be some resolution of an affine complete rational surface singularity.  Then the dimension of the Ext groups between the simples in $\Per$ are
\[
\begin{array}{lll}
\scriptstyle \ext^1(S_i,S_j)=(E_i\cdot E_j)_+ &\scriptstyle\ext^2(S_i,S_j)=(-1-E_i\cdot E_j)_+&\scriptstyle\ext^3(S_i,S_j)=0\\
\scriptstyle\ext^1(S_\star,S_\star)=0 &\scriptstyle\ext^2(S_\star,S_\star)=e-2&\scriptstyle\ext^3(S_\star,S_\star)=0\\
\scriptstyle\ext^1(S_\star,S_i)=-E_i\cdot Z_f&\scriptstyle\ext^2(S_\star,S_i)=0&\scriptstyle\ext^3(S_\star,S_i)=0\\
\scriptstyle\ext^1(S_i,S_\star)={\scriptsize\left\{  \begin{array}{c} ((Z_K-Z_f)\cdot E_i)_+\\ 1-Z_f\cdot E_i\end{array}\right.}&\scriptstyle\ext^2(S_i,S_\star)={\scriptsize\left\{  \begin{array}{c} ((Z_K-Z_f)\cdot E_i)_-\\ 1\end{array}\right.}&\scriptstyle\ext^3(S_i,S_\star)={\scriptsize\left\{  \begin{array}{c} -E_i^2-2\\ 0\end{array}\right.}
\end{array}
\]
where in the split for $\ext^t(S_i,S_\star)$ the bottom option corresponds to when $E_i$ is a {\rm ($-1$)}-curve, the top option when $E_i$ is a non-{\rm ($-1$)}-curve. All higher ext groups are zero.
\end{theorem}
\begin{proof}
We start by computing $\ext^t(S_i,S_j)=\ext^t(\c{O}_{E_i},\c{O}_{E_j})$.  Taking the short exact sequence
\[
0\to \c{O}_{\w{X}}(-E_i)\to \c{O}_{\w{X}}\to\c{O}_{E_i}\to 0
\]
and applying $\Hom(-,\c{O}_{E_j})$ gives
\[
0\to
\Hom(\c{O}_{E_i},\c{O}_{E_j})\to
\H^0(\c{O}_{E_j})\to
\H^0(\c{O}_{E_j}(E_i\cdot E_j))
\to \Ext^1(\c{O}_{E_i},\c{O}_{E_j})\to
\H^1(\c{O}_{E_j})=0
\]
and 
\[
0=\H^1(\c{O}_{E_j})\to \H^1(\c{O}_{E_j}(E_i\cdot E_j))\to \Ext^2(\c{O}_{E_i},\c{O}_{E_j})\to \H^2(\c{O}_{E_j})=0
\]
with $\ext^t(\c{O}_{E_i},\c{O}_{E_j})=0$ for all $t\geq 3$.  From the second exact sequence
\[
\ext^2(\c{O}_{E_i},\c{O}_{E_j})=\h^1(\c{O}_{E_j}(E_i\cdot E_j))= \left\{\begin{array}{cc}0 & i\neq j\\ -1-E_i\cdot E_i & i=j.\end{array}\right. 
\]
From the first exact sequence, if $i=j$ then $\H^0(\c{O}_{E_i}(E_i\cdot E_j))=0$ forces $\ext^1(\c{O}_{E_i},\c{O}_{E_j})=0$.  If $i\neq j$ then $\hom(\c{O}_{E_i},\c{O}_{E_j})=0$ and so 
\[
\ext^1(\c{O}_{E_i},\c{O}_{E_j})= \h^0(\c{O}_{E_j}(E_i\cdot E_j))-\h^0(\c{O}_{E_j})=E_i\cdot E_j.
\]
We now compute  $\ext^t(S_\star,S_\star)=\ext^t(\c{O}_{Z_f},\c{O}_{Z_f})$.  First note that $\h^1(\c{O}_{Z_f}(-Z_K))=0$ follows immediately from taking sections of the sequence
\[
0\to \c{O}_{\w{X}}(-Z_f-Z_K)\to \omega\to\c{O}_{Z_f}(-Z_K)\to 0,
\]
since $\H^1(\omega)=0$ by Grauert-Riemenschneider vanishing.  Further we know by \cite[3.4]{Art66} that $\h^0(\c{O}_{Z_f})=1$ and $\h^1(\c{O}_{Z_f})=0$.  Thus applying $\Hom(-,\c{O}_{Z_f})$ to the short exact sequence
\[
0\to \c{O}_{\w{X}}(-Z_f)\to \c{O}_{\w{X}}\to\c{O}_{Z_f}\to 0
\]
gives
\[
0\to \Hom(\c{O}_{Z_f},\c{O}_{Z_f})\to \H^0(\c{O}_{Z_f})\to \H^0(\c{O}_{Z_f}(Z_f))\to \Ext^1(\c{O}_{Z_f},\c{O}_{Z_f}) \to \H^1(\c{O}_{Z_f})=0
\]
and 
\[
0=\H^1(\c{O}_{Z_f})\to \H^1(\c{O}_{Z_f}(Z_f))\to \Ext^2(\c{O}_{Z_f},\c{O}_{Z_f})\to \H^2(\c{O}_{Z_f})=0
\]
with $\ext^t(\c{O}_{Z_f},\c{O}_{Z_f})=0$ for all $t\geq 3$.  By Serre duality on $Z_f$ we know that $\h^0(\c{O}_{Z_f}(Z_f))=\h^1(\c{O}_{Z_f}(-Z_K))=0$ and so the first exact sequence shows that $\ext^1(\c{O}_{Z_f},\c{O}_{Z_f})=0$.  Now using Serre duality on $Z_f$ the second short exact sequence gives
\[
\ext^2(\c{O}_{Z_f},\c{O}_{Z_f})=\h^1(\c{O}_{Z_f}(Z_f))=\h^0(\c{O}_{Z_f}(-Z_K)).
\]
But  
\[
\h^0(\c{O}_{Z_f}(-Z_K))=\h^0(\c{O}_{Z_f}(-Z_K))-\h^1(\c{O}_{Z_f}(-Z_K))=\chi_{}(\c{O}_{Z_f}(-Z_K))
\]
and by Riemann--Roch on $Z_f$  
\[
\chi_{}(\c{O}_{Z_f}(-Z_K))=-Z_f\cdot Z_K+1=-1-Z_f\cdot Z_f
\] 
as required.\\
Now consider $\ext^t(S_\star,S_i)$.  Due to the shift in the simples, $\ext^t(S_\star,S_i)=\ext^{t+1}(\c{O}_{Z_f},\c{O}_{E_i}(-1))$.  Simply applying $\Hom(-,\c{O}_{E_i}(-1))$ to the short exact sequence
\[
0\to \c{O}_{\w{X}}(-Z_f)\to \c{O}_{\w{X}}\to\c{O}_{Z_f}\to 0
\]
gives 
\[
0=\H^1(\c{O}_{E_i}(-1))\to \H^1(\c{O}_{E_i}(-1+E_i\cdot Z_f))\to \Ext^2(\c{O}_{Z_f},\c{O}_{E_i}(-1)) \to \H^2(\c{O}_{E_i}(-1))=0
\]
and
\[
\to \H^2(\c{O}_{E_i}(-1+E_i\cdot Z_f))=0\to \Ext^3(\c{O}_{Z_f},\c{O}_{E_i}(-1))\to \H^3(\c{O}_{E_i}(-1))=0.
\]
with all higher ext groups vanishing.  From the second sequence $\ext^2(S_\star,S_i)=0$,  whereas the first shows that
\[
 \ext^1(S_\star,S_i)=\ext^2(\c{O}_{Z_f},\c{O}_{E_i}(-1))=\h^1(\c{O}_{E_i}(-1+E_i\cdot Z_f))
\]
which equals $\h^0(\c{O}_{E_i}(-1-E_i\cdot Z_f))=-E_i\cdot Z_f$.\\
Finally we consider $\ext^t(S_i,S_\star)=\ext^{t-1}(\c{O}_{E_i}(-1),\c{O}_{Z_f})$.  Applying $\Hom(-,\c{O}_{Z_f})$ to the exact sequence
\[
0\to \c{O}_{\w{X}}(-E_i-D_i)\to \c{O}_{\w{X}}(-D_i)\to\c{O}_{E_i}(-1)\to 0
\]
gives
\begin{multline*}
0\to \Hom(\c{O}_{E_i}(-1),\c{O}_{Z_f})\to \H^0(\c{O}_{Z_f}(D_i))\to \H^0(\c{O}_{Z_f}(D_i+E_i)) \to \Ext^1(\c{O}_{E_i}(-1),\c{O}_{Z_f}) \\ \to \H^1(\c{O}_{Z_f}(D_i))\to \H^1(\c{O}_{Z_f}(D_i+E_i)) \to \Ext^2(\c{O}_{E_i}(-1),\c{O}_{Z_f})\to 0
\end{multline*}
with all higher terms zero.  By summing dimensions
\begin{multline}
\ext^1(S_i,S_\star)-\ext^2(S_i,S_\star)+\ext^3(S_i,S_\star)=\chi (\c{O}_{Z_f}(D_i))-\chi (\c{O}_{Z_f}(E_i+D_i)) =-E_i\cdot Z_f.\label{altsum}
\end{multline}
We now split into cases --- firstly assume that $E_i$ is not a ($-1$)-curve.  By Serre duality $\ext^1(S_i,S_\star)$ and $\ext^2(S_i,S_\star)$ are 
\[
\ext^2(\c{O}_{Z_f},\c{O}_{E_i}(-1-Z_K\cdot E_i))\quad \t{and}\quad \ext^1(\c{O}_{Z_f},\c{O}_{E_i}(-1-Z_K\cdot E_i)) 
\]
respectively.  Now applying $\Hom(-,\c{O}_{E_i}(-1-Z_K\cdot E_i))$ to the short exact sequence
\[
0\to \c{O}_{\w{X}}(-Z_f)\to \c{O}_{\w{X}}\to\c{O}_{Z_f}\to 0
\]
gives
\begin{multline*}
0\to \Hom(\c{O}_{Z_f},\c{O}_{E_i}(-1-Z_K\cdot E_i))\to \H^0(\c{O}_{E_i}(-1-Z_K\cdot E_i)) \to \H^0(\c{O}_{E_i}(-1+(Z_f-Z_K)\cdot E_i))\\ \to \Ext^1(\c{O}_{Z_f},\c{O}_{E_i}(-1-Z_K\cdot E_i)) \to \H^1(\c{O}_{E_i}(-1-Z_K\cdot E_i))=0
\end{multline*}
and
\begin{multline*}
0=\H^1(\c{O}_{E_i}(-1-Z_K\cdot E_i))\to \H^1(\c{O}_{E_i}(-1+(Z_f-Z_K)\cdot E_i))\to \Ext^2(\c{O}_{Z_f},\c{O}_{E_i}(-1-Z_K\cdot E_i))\to 0
\end{multline*}
since $-Z_K\cdot E_i\geq 0$ ($E_i$ is not a ($-1$)-curve). The second exact sequence shows that 
\begin{multline*}
\ext^2(\c{O}_{Z_f},\c{O}_{E_i}(-1-Z_K\cdot E_i))=\h^1(\c{O}_{E_i}(-1+(Z_f-Z_K)\cdot E_i)) =((Z_K-Z_f)\cdot E_i)_+.
\end{multline*}
and by the first exact sequence we have a surjection
\[
\C{((Z_K-Z_f)\cdot E_i)_-}\to \Ext^1(\c{O}_{Z_f},\c{O}_{E_i}(-1-Z_K\cdot E_i))\to 0
\]
which shows that $\ext^2(S_i,S_\star)\leq((Z_K-Z_f)\cdot E_i)_-$.  Thus if $(Z_K-Z_f)\cdot E_i\geq 0$ then $\ext^2(S_i,S_\star)=0$, so we may assume that $(Z_K-Z_f)\cdot E_i< 0$ in which case (by definition of $Z_f$) necessarily we must have $Z_K\cdot E_i\leq -1$.  This means that the first exact sequence reduces to  
\begin{multline*}
0\to \Hom(\c{O}_{Z_f},\c{O}_{E_i}(-1-Z_K\cdot E_i))\to \C{-Z_K\cdot E_i}\to \C{((Z_K-Z_f)\cdot E_i)_-}\\ \to \Ext^1(\c{O}_{Z_f},\c{O}_{E_i}(-1-Z_K\cdot E_i))\to 0.
\end{multline*}
But composing the surjection $\c{O}_{Z_f}\twoheadrightarrow\c{O}_{E_i}$ with a basis of $\Hom(\c{O}_{E_i},\c{O}_{E_i}(-1-Z_K\cdot E_i))$ (which has dimension $-Z_K\cdot E_i$) 
we see that $\hom(\c{O}_{Z_f},\c{O}_{E_i}(-1-Z_K\cdot E_i))\geq -Z_K\cdot E_i$.  By the above exact sequence equality holds, and so by summing dimensions we conclude that
\[
\ext^1(\c{O}_{Z_f},\c{O}_{E_i}(-1-Z_K\cdot E_i))=((Z_K-Z_f)\cdot E_i)_-.
\]
By (\ref{altsum}) we thus obtain
\begin{multline*}
\ext^3(S_i,S_\star)=-Z_f\cdot E_i+((Z_K-Z_f)\cdot E_i)_--((Z_K-Z_f)\cdot E_i)_+ =-Z_K\cdot E_i=-E_i^2-2,
\end{multline*}
finishing the proof when $E_i$ is not a ($-1$)-curve. \\
Finally, consider the case where $E_i$ is a ($-1$)-curve, then the statement is well-known if $E_i$ is the only curve in the exceptional locus (in which case $Z_f\cdot E_i=E_i\cdot E_i=-1$), thus we may assume that $E_i$ intersects some other exceptional curve.  In this case $Z_f\cdot E_i=0$ by Lemma~\ref{inequality}.  Now by Serre duality $\ext^1(S_i,S_\star)$ and $\ext^2(S_i,S_\star)$ are 
\[
\ext^2(\c{O}_{Z_f},\c{O}_{E_i}(-2))\quad \t{and}\quad \ext^1(\c{O}_{Z_f},\c{O}_{E_i}(-2)) 
\]
respectively, so applying $\Hom(-,\c{O}_{E_i}(-2))$ to the exact sequence
\[
0\to \c{O}_{\w{X}}(-Z_f)\to \c{O}_{\w{X}}\to\c{O}_{Z_f}\to 0
\]
gives
\begin{multline*}
\H^0(\c{O}_{E_i}(-2+Z_f\cdot E_i))\to \Ext^1(\c{O}_{Z_f},\c{O}_{E_i}(-2))\to \H^1(\c{O}_{E_i}(-2))\to\\ \to  \H^1(\c{O}_{E_i}(-2+Z_f\cdot E_i))\to \Ext^2(\c{O}_{Z_f},\c{O}_{E_i}(-2)) \to 0,
\end{multline*}
which since $Z_f\cdot E_i=0$ is just
\[
0\to \Ext^1(\c{O}_{Z_f},\c{O}_{E_i}(-2))\to \C{}\to \C{}\to\Ext^2(\c{O}_{Z_f},\c{O}_{E_i}(-2)) \to 0.
\]
Thus $\ext^1(S_i,S_\star)=\ext^2(S_i,S_\star)\leq 1$.  To see that both are precisely one we exhibit a non-zero map from $\c{O}_{E_i}(-1)$ to $\c{O}_{Z_f}$, then $\ext^1(S_i,S_\star)=\hom(\c{O}_{E_i}(-1),\c{O}_{Z_f})\neq 0$.  But since $E_i$ intersects some other curve, $Z_f-E_i>0$ so there is an exact sequence
\[
0\to \c{O}_{E_i}(-(Z_f-E_i))\to\c{O}_{Z_f}\to\c{O}_{Z_f-E_i}\to 0
\]
which since $Z_f\cdot E_i=0$ is simply
\[
0\to \c{O}_{E_i}(-1)\to\c{O}_{Z_f}\to\c{O}_{Z_f-E_i}\to 0,
\]
providing us with the required non-zero map. Lastly by (\ref{altsum}) 
\[
\ext^3(S_i,S_\star)=-Z_f\cdot E_i+1-1=-Z_f\cdot E_i=0,
\]
finishing the proof.
\end{proof}

\begin{corollary}\label{main}
Let $\w{X}\rightarrow \Spec R$ be some resolution of an affine complete rational surface singularity.  Then $\End_{\w{X}}(\c{O}_{\w{X}}\oplus \c{M}_{I})$ can be written as a quiver with relations as follows: for every exceptional curve $E_i$ associate a vertex labelled $i$, and also associate a vertex $\star$ corresponding to $\c{O}_{\w{X}}$.  Then the number of arrows and relations between the vertices is given as follows:
\[
\begin{array}{c|cc}
&\mbox{Number of arrows}&\mbox{Number of relations}\\ \hline
i\rightarrow j&(E_i\cdot E_j)_+  & (-1-E_i\cdot E_j)_+\\
\star\rightarrow\star&0&-Z_K\cdot Z_f+1=-1-Z_f\cdot Z_f\\
i\rightarrow \star&-E_i\cdot Z_f&0\\
\star\rightarrow i&\left\{  \begin{array}{c} ((Z_K-Z_f)\cdot E_i)_+\\ 1-Z_f\cdot E_i\end{array}\right.&\left\{  \begin{array}{c} ((Z_K-Z_f)\cdot E_i)_-\\ 1\end{array}\right.
\end{array}
\]
where in the split for $\star\rightarrow i$, the bottom option corresponds to when $E_i$ is a {\rm ($-1$)}-curve, the top option when $E_i$ is a non-{\rm ($-1$)}-curve. 
\end{corollary}
\begin{proof}
Denote $A:=\End_{\w{X}}(\c{O}_{\w{X}}\oplus \c{M}_{I})$. In the conventions here right modules are the same as representations of quivers, so to write $A$ as a quiver with relations we need to take the simple right $A$-modules and calculate the dimensions of the ext groups between them.  In the above the simples which were denoted by $S$ are left modules (since the functor $\Hom_{\w{X}}(\c{O}_{\w{X}}\oplus \c{M}_{I},-)$ has image in left modules) and so we need to reverse the order of the simples,  i.e.
\[
\begin{array}{c}
\t{number of arrows in $A$ }i\rightarrow j = \ext^1_A(S_j,S_i)\\
\t{number of relations in $A$ } i\rightarrow j = \ext^2_A(S_j,S_i)
\end{array}
\]
The fact that this gives a presentation of the algebra is mostly well-known.  The statement on the $\ext^1$ is always true whereas for the statement on $\ext^2$ we are using the fact that we are in the formal case so our path algebras are complete; see for example \cite[3.4(b)]{BIKS}.
\end{proof}

\begin{corollary}\label{gldim1}
Let $\w{X}\rightarrow \Spec R$ be some resolution of an affine complete rational surface singularity, then 
\[
\gl \End_{\w{X}}(\c{O}_{\w{X}}\oplus \c{M}_{I})=\left\{\begin{array}{cl} 3& \mbox{if there exists $E_i$ with $E_i^2<-2$}\\ 2 & \mbox{else.}\end{array}\right. 
\]
\end{corollary}
\begin{proof}
By Theorem~\ref{extcalculation}  $\Ext^t(X,Y)=0$ for all $t\geq 3$ and all simples $X$ and $Y$, except possibly the case $\Ext^3(S_i,S_\star)$ which is zero unless $E_i^2<-2$.
 \end{proof}

We now relate the above to the work of Wunram so below we restrict our attention to the minimal resolution $\pi:\hat{X}\to\Spec R$.  Recall the following:

\begin{definition}[\cite{Wunram_generalpaper}]\label{specialdef}
For a given CM module $M$ of $R$ define $\c{M}:=\pi^*M/{\rm torsion}$ to be the corresponding vector bundle on $\hat{X}$.  A CM $R$-module $M$ is called \emph{special} if $\H^1(\c{M}^\vee)=0$.
\end{definition}

By Van den Bergh's construction it is immediate that $\c{O}_{\hat{X}}\oplus \c{M}_{I}$ is equal to the sum, over all special indecomposable CM $R$-modules, of the corresponding vector bundles $\c{M}$ defined in \ref{specialdef}.  

\begin{lemma}\label{up=down}
Let $\hat{X}$ be the minimal resolution, then the natural map 
\[
\End_{\hat{X}}(\c{O}_{\hat{X}}\oplus \c{M}_{I})\rightarrow \End_R(\oplus M)
\] is an isomorphism, where the right hand sum is taken over all indecomposable special CM $R$-modules.
\end{lemma}
\begin{proof}
It is well-known that if $M$ is a CM $R$-module then $\pi_{*}\c{M}=M$ \cite[2.2]{Esnault}.  Thus taking global sections gives the natural map, which is an isomorphism away from the unique singular point.  We know that $\End_R(\oplus M)$  is reflexive since it is CM, thus if we prove that  $\End_{\hat{X}}(\c{O}_{\hat{X}}\oplus \c{M}_{I})$ is reflexive then it follows that the map is an isomorphism.  But now each $\Hom_{\hat{X}}(\c{M}_1,\c{M}_2)\cong \pi_{*}(\c{M}_1^{\vee}\otimes \c{M}_2)$ and trivially  $\c{M}_1^{\vee}\otimes \c{M}_2$ is locally free.  Further
\[
\H^1((\c{M}_1^{\vee}\otimes \c{M}_2)^\vee\otimes \omega)= \H^1(\c{M}_2^\vee\otimes \c{M}_1\otimes \omega)
\]
equals zero since $\H^1(\c{M}_2^\vee\otimes\omega)=0$ and also $\c{M}_{1}$ is generated by global sections.  Hence by \cite[2.1]{Esnault} we see that each $\Hom_{\hat{X}}(\c{M}_1,\c{M}_2)\cong \pi_{*}(\c{M}_1^{\vee}\otimes \c{M}_2)$ is reflexive, thus $\End_{\hat{X}}(\c{O}_{\hat{X}}\oplus \c{M}_{I})$ is reflexive.
\end{proof}

Consequently there is a version of Corollary~\ref{main} for $\End_R(\oplus M)$ and so this proves Theorem~\ref{intromain}.  Since in the quotient case the indecomposable special CM modules are precisely $(\rho\otimes \C{}[[x,y]])^G$ as $\rho$ runs over the special irreducible representations, this also proves Corollary~\ref{GLintro}.

Combining Lemma~\ref{up=down} and Corollary~\ref{gldim1} gives the following improvement of \cite[2.15]{Iyama_Wemyss_specials}.
\begin{corollary}\label{gldim2}
Let $R$ be an affine complete rational surface singularity and set $A:=\End_{R}(\oplus M)$ where the sum is taken over all indecomposable special CM $R$-modules.  Then 
\[
\gl A=\left\{\begin{array}{cl} 2& \mbox{if $R$ is Gorenstein}\\ 3 & \mbox{else.}\end{array}\right. 
\]
When $R$ is Gorenstein all simple left $A$-modules and all simple right $A$-modules have projective dimension 2.  When $R$ is not Gorenstein all simple right $A$-modules have projective dimension 2 except the simple corresponding to $\star$, which has projective dimension 3.  As left $A$-modules, the projective dimension of the simples at $\star$ and all curves corresponding to $(-2)$-curves have projective dimension 2, whereas all other simples have projective dimension 3.
\end{corollary}
\begin{proof}
The statement on the global dimension follows immediately from Corollary~\ref{gldim1} since $\End_{R}(\oplus M)\cong \End_{\hat{X}}(\c{O}_{\hat{X}}\oplus \c{M}_{I})$ by Lemma~\ref{up=down} and on the minimal resolution the configurations containing only $(-2)$-curves occur precisely when $R$ is Gorenstein. \\
Now the simples $S$ in Theorem~\ref{extcalculation} are left $A=\End_{R}(\oplus M)$-modules with
\[
\pd_A (S_i)=\t{min}\{ t\geq 0\mid\Ext^t(S_i,X)=0 \t{ for all simples }X  \}-1
\]
and so the statement on left modules follows immediately by inspecting the ext groups in Theorem~\ref{extcalculation}.  Now denote $T_{i}$ to be the simple right $A$-module corresponding to $M_i$ in the decomposition of $M$, then
\[
\pd_A (T_i)=\t{min}\{ t\geq 0\mid\Ext^t(X,S_i)=0 \t{ for all simples }X  \}-1
\]
and so the result follows from Theorem~\ref{extcalculation}.
\end{proof}

We finish this section with the following trivial but convenient lemma which reduces the calculation of the quiver to simply adding arrows to a certain base quiver, as is true for reconstruction algebras of type $A$ \cite{Wemyss_reconstruct_A}.
\begin{lemma}\label{basecase}
Suppose $E=\{ E_i\}$ forms the exceptional curves on a minimal resolution of some affine complete rational surface singularity.  Associate the fundamental cycle $Z_f=\sum r_iE_i$ and canonical cycle $Z_K^E$.  Suppose  $F=\{ F_i\}$ forms the exceptional curves on a minimal resolution of another affine complete rational surface singularity with the same dual graph \textbf{and} fundamental cycle, such that $-F_i^2\leq -E_i^2$ for all $i$.  Then\\
\t{(i)} $-Z_f\cdot E_i=-Z_f\cdot F_i+(-E_i^2+F_i^2)$\\
\t{(ii)} $(Z_K^E-Z_f)\cdot E_i=(Z_K^F-Z_f)\cdot F_i$\\
Consequently the quiver for the curve system $E$ is obtained from the quiver of the curve system $F$ by adding $-E_i^2+F_i^2$ extra arrows $i\rightarrow \star$ for every curve $E_i$.  
\end{lemma}
\begin{proof}
(i) If $-F_i^2= -E_i^2$ then there is nothing to prove, so we can assume that $-F_i^2< -E_i^2$.  Since both $F$ and $E$ have the same $Z_f$, by definition it is true that  $-Z_f\cdot E_i=-Z_f\cdot F_i+r_i(-E_i^2+F_i^2)$.  However the point is that by combinatorics on rational surfaces, the condition $-F_i^2< -E_i^2$ forces $r_i=1$ (see e.g. \cite[3.9]{TT}) and so (i) follows.  The remaining part (ii) is now trivial.
\end{proof}

\section{The Representation Theoretic Method}\label{Section:RepMethod}
Here we compute $\End_{\C{}[[x,y]]^G}( \oplus_{\rho\,\,\t{special}} (\rho\otimes \C{}[[x,y]])^G)$ directly, without assuming any of the geometry.  As stated in the introduction, combining this with Theorem~\ref{main} and Lemma~\ref{up=down}  gives a method to recover the dual graph directly from the representation theory, and also provides us with more information than the non-explicit proof.

Denote $R=\C{}[[x,y]]^G$, where $G$ is some small finite subgroup of
$\GL(2,\C{})$.  Note that the AR quiver of the category of CM
$R$-modules (i.e. all irreducible maps between the CM modules)
coincides with the McKay quiver by a result of Auslander
\cite{Auslander_AR=McKay}, and further all such quivers for small
finite subgroups of $\GL(2,\C{})$ were classified in
\cite{AR_McKayGraphs}.  Below let $\c{S}$ be the set of indecomposable special CM $R$-modules.

For two indecomposable special CM $R$-modules $M$ and $N$ we wish to determine the number of
arrows from $M$ to $N$ in the quiver of the reconstruction algebra.  Because everything is graded, this is just the dimension
of the space of morphisms from $M$ to $N$ which don't factor through
a special CM module via maps of strictly smaller positive degree.

Proceed as follows:
\begin{itemize}
\item[(i)] In the AR quiver, begin by writing a 1 at the position corresponding to $M$ and then define for every CM module $V$
\[
\mu_V^{(0)}:=\lambda_V^{(0)}:=\left\{  \begin{array}{cl} 1 & \t{ if } V=M\\0&\t{ else.} \end{array}\right. 
\]
\item[(ii)] Consider all arrows out of $M$ in the AR quiver. For convenience we call the heads of these arrows the first-step vertices.  For every CM module $V$ define
\[
\lambda_V^{(1)}=\left\{  \begin{array}{cl} 1 & \t{ if } V \t{ is a first-step vertex}\\0&\t{ else} \end{array}\right.
\]
and then define
\[
\mu_V^{(1)}=\left\{  \begin{array}{cl}0&\t{ if } V \in\c{S}\\
\lambda_V^{(1)} & \t{ else.}
\end{array}\right. 
\]
In the AR quiver, at every first-step vertex $M_1$ write the corresponding number 
$\lambda_{M_1}^{(1)}$ and then circle the first-step vertices which
belong to $\c{S}$.
\item[(iii)] Now consider all arrows out of the first-step vertices.  The heads of these arrows are called the second-step vertices.  For every CM module $V$ define
\[
\lambda_{V}^{(2)}=\left\{ \begin{array}{cl}\t{max}\{0,-\mu_{\tau
(V)}^{(0)} +\sum_{a:L\rightarrow V}\mu_{L}^{(1)} \}& \t{ if }V\t{ is
a second-step vertex} \\ 0 & \t{ else} \end{array}\right.
\]
and then define
\[
\mu_{V}^{(2)}=\left\{ \begin{array}{cl}0& \t{ if }V\in\c{S}
\\ \lambda_V^{(2)} & \t{ else.}
\end{array}\right. 
\]
At every second-step vertex $M_2$ write the corresponding number
$\lambda_{M_2}^{(2)}$ and then circle the second-step vertices which
belong to $\c{S}$.
\item[(iv)] Next consider all arrows in the AR quiver out of the second-step vertices.  The heads of these arrows are called the third-step vertices.  For every CM module $V$ define
\[
\lambda_{V}^{(3)}=\left\{ \begin{array}{cl}\t{max}\{0, -\mu_{\tau
(V)}^{(1)} +\sum_{a:L\rightarrow V}\mu_{L}^{(2)} \}& \t{ if }V\t{ is
a third-step vertex} \\ 0 & \t{ else} \end{array}\right.
\]
and
\[
\mu_{V}^{(3)}=\left\{ \begin{array}{cl}0& \t{ if }V\in\c{S}
\\ \lambda_V^{(3)} & \t{ else.}
\end{array}\right.
\]
At every third-step vertex $M_3$ write the corresponding number $\lambda_{M_3}^{(3)}$ and
then circle the third-step vertices which belong to $\c{S}$.
\item[(v)] Continue in this fashion.  For any $V$ and any $i$, $\lambda_V^{(i)}$
records the dimension of the space of maps of degree $i$ from $M$ to
$V$ which don't factor through $\c{S}$ via maps of strictly smaller
positive degree (for the proof, see below).  Consequently the
dimension of the space of maps from $M$ to $N$ which don't factor
through $\c{S}$ via maps of strictly smaller positive degree is
simply $\sum_{i\geq 1}\lambda_N^{(i)}$.
\end{itemize}
\begin{remark}
The proof below is similar to \cite[\S 4]{Iyama_Wemyss_specials} (which itself was inspired by knitting), where there $\c{S}=\{ R\}$ was taken to
establish the dimension of the vector space $\u{\Hom}_{R}(M,N)$.
However the proof here is more subtle since now both $M$ and
$N$ belong to $\c{S}$ and so $\u{\Hom}_{\c{S}}(M,N)=0$, making it harder to use functorial proofs. 
\end{remark}
\noindent
\emph{Proof of algorithm:}\\
It is clear that the start of the algorithm is correct and so we
only need to verify the general induction step.  Thus assume that the result
is true for smaller $n$.   Consider the AR quiver $Q_{AR}$ of the
category of CM $\C{}[[x,y]]^G$-modules.  By \cite{AR_McKayGraphs}
this is just $\mathbb{Z}\Delta/G$ where $\Delta$ is some extended
Dynkin quiver and $G$ is some group of automorphisms; in fact
\[
\mathbb{Z}\Delta\rightarrow Q_{AR}
\]
is a covering.  Fix $M$ to be in degree 0 --- it is more convenient to
work with $\mathbb{Z}\Delta$ since there the grading is evident.
Consider the mesh category $\c{C}:=k(\mathbb{Z}\Delta)$; in the
language of \cite{Iyama_tau1}, $\c{C}$ is a $\tau$-category. Now
define $\c{S}_n$ to be all those objects of $\c{C}$ lying in degrees
between 1 and $n-1$ (inclusive) which belong to $\c{S}$, and
consider the quotient category $\c{C}/ [\c{S}_n]$; this too is a
$\tau$-category. Since the AR quiver records all irreducible maps,
it is clear by construction that if $V$ is any CM module then the
dimension of the space of maps of degree $n$ between $M$ and $V$
which don't factor through an object of $\c{S}$ via maps of strictly
smaller positive degree is
\[
\dim_k \Hom_{\c{C}/ [\c{S}_n]}(M,V)
\]
where the $M$ sits in degree 0 and the $V$ sits in degree $n$.  By
the theory of ladders in $\tau$-categories \cite{Iyama_tau1} this
can be calculated explicitly using a recursion formula (see e.g.
\cite[4.5]{Iyama_Wemyss_specials}).  Now the AR quiver of $\c{C}/
[\c{S}_n]$ is just the quiver of $\c{C}$ with the vertices in
$\c{S}_n$ deleted, and furthermore by inductive hypothesis (and
construction of the smaller $\c{C}/[\c{S}_{n-1}]$) any term in the
recursion at a previous vertex is given by the $\mu$ associated to
that vertex.  Consequently the recursion formula in
\cite[4.5]{Iyama_Wemyss_specials} gives
\[
\dim_k \Hom_{\c{C}/ [\c{S}_n]}(M,V)=\left\{
\begin{array}{cl}\t{max}\{ 0, -\mu_{\tau (V)}^{(n-2)}
+\sum_{a:L\rightarrow V}\mu_{L}^{(n-1)} \}& \t{ if }V\t{ is a step
$n$ vertex} \\ 0 & \t{ else} \end{array}\right.
\]
which by definition is $\lambda_V^{(n)}$. Thus indeed
$\lambda_V^{(n)}$ records the dimension of the space of maps between
$M$ and $V$ of degree $n$ which don't factor through an object of
$\c{S}$ via maps of strictly smaller positive degree.

\medskip

We now use the above algorithm to re-prove Theorem~\ref{main} for some examples of groups $\mathbb{I}_m$ (see \S\ref{groups} for notation).  By \cite{AR_McKayGraphs} the AR quiver of $\C{}[[x,y]]^{\mathbb{I}_m}$ is
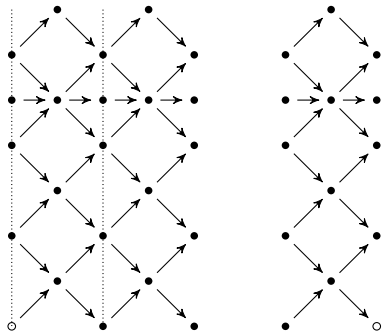
\[
\begin{tikzpicture}[xscale=0.6,yscale=0.6]
\node (A1) at (1,0) [vertex] {};
\node (A2) at (3,0) [vertex] {};
\node (A3) at (7,0) [vertex] {};
\node (B1) at (0,-1) [vertex] {};
\node (B2) at (2,-1) [vertex] {};
\node (B3) at (4,-1) [vertex] {};
\node (B4) at (6,-1) [vertex] {};
\node (B5) at (8,-1) [vertex] {};
\node (C1) at (0,-2) [vertex] {};
\node (C2) at (1,-2) [vertex] {};
\node (C3) at (2,-2) [vertex] {};
\node (C4) at (3,-2) [vertex] {};
\node (C5) at (4,-2) [vertex] {};
\node (C6) at (6,-2) [vertex] {};
\node (C7) at (7,-2) [vertex] {};
\node (C8) at (8,-2) [vertex] {};
\node (D1) at (0,-3) [vertex] {};
\node (D2) at (2,-3) [vertex] {};
\node (D3) at (4,-3) [vertex] {};
\node (D4) at (6,-3) [vertex] {};
\node (D5) at (8,-3) [vertex] {};
\node (E1) at (1,-4) [vertex] {};
\node (E2) at (3,-4) [vertex] {};
\node (E3) at (7,-4) [vertex] {};
\node (F1) at (0,-5) [vertex] {};
\node (F2) at (2,-5) [vertex] {};
\node (F3) at (4,-5) [vertex] {};
\node (F4) at (6,-5) [vertex] {};
\node (F5) at (8,-5) [vertex] {};
\node (G1) at (1,-6) [vertex] {};
\node (G2) at (3,-6) [vertex] {};
\node (G3) at (7,-6) [vertex] {};
\node (H1) at (0,-7) [cvertex] {};
\node (H2) at (2,-7) [vertex] {};
\node (H3) at (4,-7) [vertex] {};
\node (H4) at (6,-7) [vertex] {};
\node (H5) at (8,-7) [cvertex] {};
\draw[->] (B1) -- (A1);
\draw[->] (A1) -- (B2);
\draw[->] (B2) -- (A2);
\draw[->] (A2) -- (B3);
\draw[->] (B4) -- (A3);
\draw[->] (A3) -- (B5);

\draw[->] (B1) -- (C2);
\draw[->] (C2) -- (B2);
\draw[->] (B2) -- (C4);
\draw[->] (C4) -- (B3);
\draw[->] (B4) -- (C7);
\draw[->] (C7) -- (B5);

\draw[->] (C1) -- (C2);
\draw[->] (C2) -- (C3);
\draw[->] (C3) -- (C4);
\draw[->] (C4) -- (C5);
\draw[->] (C6) -- (C7);
\draw[->] (C7) -- (C8);

\draw[->] (D1) -- (C2);
\draw[->] (C2) -- (D2);
\draw[->] (D2) -- (C4);
\draw[->] (C4) -- (D3);
\draw[->] (D4) -- (C7);
\draw[->] (C7) -- (D5);

\draw[->] (D1) -- (E1);
\draw[->] (E1) -- (D2);
\draw[->] (D2) -- (E2);
\draw[->] (E2) -- (D3);
\draw[->] (D4) -- (E3);
\draw[->] (E3) -- (D5);

\draw[->] (F1) -- (E1);
\draw[->] (E1) -- (F2);
\draw[->] (F2) -- (E2);
\draw[->] (E2) -- (F3);
\draw[->] (F4) -- (E3);
\draw[->] (E3) -- (F5);

\draw[->] (F1) -- (G1);
\draw[->] (G1) -- (F2);
\draw[->] (F2) -- (G2);
\draw[->] (G2) -- (F3);
\draw[->] (F4) -- (G3);
\draw[->] (G3) -- (F5);

\draw[->] (H1) -- (G1);
\draw[->] (G1) -- (H2);
\draw[->] (H2) -- (G2);
\draw[->] (G2) -- (H3);
\draw[->] (H4) -- (G3);
\draw[->] (G3) -- (H5);

\draw[densely dotted] (0,0) -- (0,-7);
\draw[densely dotted] (2,0) -- (2,-7);
\end{tikzpicture}
\]
where $\circ$ represents the free module and there are precisely $m$ repetitions of the original $\tilde{E}_8$ shown in dotted lines. The left and right hand sides of the picture are identified, and there is no twist in this AR quiver.  The AR translation $\tau$ moves each dotted segment one place to the left; $\tau^{-1}$ therefore moves it one place to the right.

\begin{example}\label{I1red}
Consider the group $\mathbb{I}_{30(b-2)+1}$ with $b\geq 3$.  By \cite[9.2]{Iyama_Wemyss_specials} the indecomposable special CM modules are $R$, $A_1$, $A_2$, $A_3$, $A_4$, $B_1$, $B_2$, $C$ and $M$ where
\[
\begin{array}{c}
A_i:=\tau^{-6i}R \mbox{ for all }1\leq i\leq 4\\
B_i:=\tau^{-10i}R \mbox{ for all }1\leq i\leq 2\\
C:=\tau^{-15}R\\
M:=\tau^{-30}R
\end{array}
\]

To determine the number of irreducible maps from $R$ to the other special CM modules we place a 1 in the position of $R$ (double circled below) and start counting:

{\tiny{
\[
\begin{array}{c}
\xymatrix@C=-4pt@R=-4pt{
&.&&.&&.&&1&&0&&0&&1&&0&&1&&0&&0&&1&&0\\
.&&.&&.&&1&&1&&0&&1&&1&&1&&1&&0&&1&&1&&0\\
.&.&.&.&.&1&1&1&0&1&1&1&0&1&1&2&1&1&0&1&1&1&0&1&1&1&0&0\\
.&&.&&1&&0&&1&&1&&1&&1&&1&&1&&1&&1&&0&&1&&0\\
&.&&1&&0&&0&&1&&1&&1&&0&&1&&1&&1&&0&&0&&1&&0 \\
.&&1&&0&&0&&0&&1&&1&&0&&0&&1&&1&&0&&0&&0&&1&&0 \\
&1&&0&&0&&0&&0&&1&&0&&0&&0&&1&&0&&0&&0&&0&&1&&0\\
{{}\drop\xycircle<4pt,4pt>{}{}\drop\xycircle<4.5pt,4.5pt>{}1}&&0&&0&&0&&0&&0&&{{}\drop\xycircle<4pt,4pt>{}1}&&0&&0&&0&&{{}\drop\xycircle<4pt,4pt>{}1}&&0&&{{}\drop\xycircle<4pt,4pt>{}0}&&0&&0&&{{}\drop\xycircle<4pt,4pt>{}1}&&0
}
\end{array}
\]}}Thus there is precisely one map from $R$ to $A_1$ (of grade 12), one map from $R$ to $B_1$ (of grade 20) and one map from $R$ to $C$ (of grade 30).

To determine the number of irreducible maps out of $A_1$, we place a 1 in the position of $A_1$ and start counting:

{\tiny{
\[
\begin{array}{c}
\xymatrix@C=-5pt@R=-4pt{
&.&&.&&.&&1&&0&&0&&1&&0&&1&&0&&0&&1&\\
.&&.&&.&&1&&1&&0&&1&&1&&1&&1&&0&&1&&1\\
.&.&.&.&.&1&1&1&0&1&1&1&0&1&1&2&1&1&0&1&1&1&0&1&1\\
.&&.&&1&&0&&1&&1&&1&&1&&1&&1&&1&&1&&1\\
&.&&1&&0&&0&&1&&1&&1&&0&&1&&1&&1&&1&\\
.&&1&&0&&0&&0&&1&&1&&0&&0&&1&&1&&1&&0\\
&1&&0&&0&&0&&0&&1&&0&&0&&0&&1&&1&&0&\\
{{}\drop\xycircle<4pt,4pt>{}{}\drop\xycircle<4.5pt,4.5pt>{}1}&&0&&0&&0&&{{}\drop\xycircle<4pt,4pt>{}0}&&0&&{{}\drop\xycircle<4pt,4pt>{}1}&\ar@{.}@<1ex>[-7,0]&0&&0&&{{}\drop\xycircle<4pt,4pt>{}0}&&1&&0&\ar@{.}@<1ex>[-7,0]&{{}\drop\xycircle<4pt,4pt>{}0}}
\end{array}
\hdots
\begin{array}{c}
\xymatrix@C=-4pt@R=-4pt{
0&&0&&1&&0&\\
&0&&1&&1&&0\\
1&1&1&0&1&1&1&0&0&\\
&1&&1&&0&&1&&0\\
1&&1&&0&&0&&1&&0\\
&1&&0&&0&&0&&1&&0\\
1&&0&&0&&0&&0&&1&&0\\
&{{}\drop\xycircle<4pt,4pt>{}1}&&0&&0&&0&&0&&1&&{{}\drop\xycircle<4pt,4pt>{}0}&
}
\end{array}
\]}}In the above picture the second special circled (i.e $A_2$) absorbs a 1 and the calculation continues, repeating the segment between the dotted lines.  Careful analysis shows that this repetition misses all the specials and keeps repeating until it reaches $R$ again, which absorbs a 1 (on the right hand side of the picture above), forcing the calculation to finish.  We deduce that there is precisely one irreducible map from $A_1$ to $A_2$, and one irreducible map from $A_1$ to $R$, and these are all the maps out of $A_1$.

The calculations for $A_3$, $A_4$, $B_1$, $B_2$ and $C$ are very similar.   The only remaining calculation which is non-trivial is determining the number of irreducible maps out of $M$.  For this begin by placing a 1 in the position of $M$ and begin to count.  Since there are no special CM modules between $M$ and $R$ this calculation between $M$ and $R$ is just the free expansion starting at $M$ (for terminology, see \cite[7.1]{Iyama_Wemyss_specials}).  This continues until it reaches $R$ (a grading of $2+60(b-3)$ away) and so \cite[9.1]{Iyama_Wemyss_specials} specifies the first two columns below, where the first circle corresponds to $R$:
\[
\begin{tikzpicture}[xscale=0.6,yscale=0.2]
\node at (0,7) {$\scriptscriptstyle 2b\t{-} 6$};
\node at (2,7) {$\scriptscriptstyle 2b\t{-} 6$};
\node at (4,7) {$\scriptscriptstyle 2b\t{-} 6$};
\node at (6,7) {$\scriptscriptstyle 2b\t{-} 5$};
\node at (8,7) {$\scriptscriptstyle b\t{-} 3$};
\node at (10,7) {$\scriptscriptstyle 2b\t{-} 6$};
\node at (12,7) {$\scriptscriptstyle 2b\t{-} 5$};
\node at (14,7) {$\scriptscriptstyle b\t{-} 3$};
\node at (16,7) {$\scriptscriptstyle 2b\t{-} 5$};
\node at (18,7) {$\scriptscriptstyle b\t{-} 2$};
\node at (1,6) {$\scriptscriptstyle 4b\t{-}12$};
\node at (3,6) {$\scriptscriptstyle 4b\t{-}12$};
\node at (5,6) {$\scriptscriptstyle 4b\t{-}11$};
\node at (7,6){$\scriptscriptstyle 3b\t{-}8$};
\node at (9,6){$\scriptscriptstyle 3b\t{-}9$};
\node at (11,6){$\scriptscriptstyle 4b\t{-}11$};
\node at (13,6){$\scriptscriptstyle 3b\t{-}8$};
\node at (15,6){$\scriptscriptstyle 3b\t{-}8$};
\node at (17,6){$\scriptscriptstyle 3b\t{-}7$};
\node at (19,6){$\scriptscriptstyle 2b\t{-}5$};
\node at (0,5) {$\scriptscriptstyle 6b\t{-}18$};
\node at (1,5) {$\scriptscriptstyle 3b\t{-}9$};
\node at (2,5) {$\scriptscriptstyle 6b\t{-}18$};
\node at (3,5) {$\scriptscriptstyle 3b\t{-}9$};
\node at (4,5) {$\scriptscriptstyle 6b\t{-}17$};
\node at (5,5) {$\scriptscriptstyle 3b\t{-}8$};
\node at (6,5) {$\scriptscriptstyle 5b\t{-}14$};
\node at (7,5) {$\scriptscriptstyle 2b\t{-}6$};
\node at (8,5) {$\scriptscriptstyle 5b\t{-}14$};
\node at (9,5) {$\scriptscriptstyle 3b\t{-}8$};
\node at (10,5) {$\scriptscriptstyle 5b\t{-}14$};
\node at (11,5) {$\scriptscriptstyle 2b\t{-}6$};
\node at (12,5) {$\scriptscriptstyle 5b\t{-}14$};
\node at (13,5) {$\scriptscriptstyle 3b\t{-}8$};
\node at (14,5) {$\scriptscriptstyle 5b\t{-}13$};
\node at (15,5) {$\scriptscriptstyle 2b\t{-}5$};
\node at (16,5) {$\scriptscriptstyle 4b\t{-}10$};
\node at (17,5) {$\scriptscriptstyle 2b\t{-}5$};
\node at (18,5) {$\scriptscriptstyle 4b\t{-}10$};
\node at (19,5) {$\scriptscriptstyle 2b\t{-}5$};
\node at (1,4) {$\scriptscriptstyle 5b\t{-}15$};
\node at (3,4) {$\scriptscriptstyle 5b\t{-}14$};
\node at (5,4) {$\scriptscriptstyle 4b\t{-}12$};
\node at (7,4) {$\scriptscriptstyle 5b\t{-}14$};
\node at (9,4) {$\scriptscriptstyle 4b\t{-}11$};
\node at (11,4) {$\scriptscriptstyle 4b\t{-}11$};
\node at (13,4) {$\scriptscriptstyle 4b\t{-}11$};
\node at (15,4) {$\scriptscriptstyle 4b\t{-}10$};
\node at (17,4) {$\scriptscriptstyle 3b\t{-}8$};
\node at (19,4) {$\scriptscriptstyle 4b\t{-}10$};
\node at (0,3) {$\scriptscriptstyle 4b\t{-}12$};
\node at (2,3) {$\scriptscriptstyle 4b\t{-}11$};
\node at (4,3) {$\scriptscriptstyle 3b\t{-}9$};
\node at (6,3) {$\scriptscriptstyle 4b\t{-}12$};
\node at (8,3) {$\scriptscriptstyle 4b\t{-}11$};
\node at (10,3) {$\scriptscriptstyle 3b\t{-}8$};
\node at (12,3) {$\scriptscriptstyle 3b\t{-}8$};
\node at (14,3) {$\scriptscriptstyle 3b\t{-}8$};
\node at (16,3) {$\scriptscriptstyle 3b\t{-}8$};
\node at (18,3) {$\scriptscriptstyle 3b\t{-}8$};
\node at (1,2) {$\scriptscriptstyle 3b\minus 8$};
\node at (3,2) {$\scriptscriptstyle 2b\minus 6$};
\node at (5,2) {$\scriptscriptstyle 3b\minus 9$};
\node at (7,2) {$\scriptscriptstyle 3b\minus 9$};
\node at (9,2) {$\scriptscriptstyle 3b\minus 8$};
\node at (11,2) {$\scriptscriptstyle 2b\minus 5$};
\node at (13,2) {$\scriptscriptstyle 2b\minus 5$};
\node at (15,2) {$\scriptscriptstyle 2b\minus 6$};
\node at (17,2) {$\scriptscriptstyle 3b\minus 8$};
\node at (19,2) {$\scriptscriptstyle 2b\minus 5$};
\node at (0,1) {$\scriptscriptstyle 2b\t{-}5$};
\node at (2,1) {$\scriptscriptstyle b\t{-}3$};
\node at (4,1) {$\scriptscriptstyle 2b\t{-}6$};
\node at (6,1) {$\scriptscriptstyle 2b\t{-}6$};
\node at (8,1) {$\scriptscriptstyle 2b\t{-}6$};
\node at (10,1) {$\scriptscriptstyle 2b\t{-}5$};
\node at (12,1) {$\scriptscriptstyle b\t{-}2$};
\node at (14,1) {$\scriptscriptstyle b\t{-}3$};
\node at (16,1) {$\scriptscriptstyle 2b\t{-}6$};
\node at (18,1) {$\scriptscriptstyle 2b\t{-}5$};
\node at (1,0) {$\scriptscriptstyle b\t{-}3$};
\node at (3,0) {$\scriptscriptstyle b\t{-}3$};
\node at (5,0) {$\scriptscriptstyle b\t{-}3$};
\node at (7,0) {$\scriptscriptstyle b\t{-}3$};
\node at (9,0) {$\scriptscriptstyle b\t{-}3$};
\node at (11,0) {$\scriptscriptstyle b\t{-}2$};
\node at (13,0) {$\scriptscriptstyle 0$};
\node at (15,0) {$\scriptscriptstyle b\t{-}3$};
\node at (17,0) {$\scriptscriptstyle b\t{-}3$};
\node at (19,0) {$\scriptscriptstyle b\t{-}2$};
\draw (1,0) ellipse (9pt and 27pt);
\draw (13,0) ellipse (7pt and 21pt);
\end{tikzpicture}
\]
\[
\begin{tikzpicture}[xscale=0.6,yscale=0.2]
\node at (0,7) {$\scriptscriptstyle b\t{-} 3$};
\node at (2,7) {$\scriptscriptstyle 2b\t{-}5$};
\node at (4,7) {$\scriptscriptstyle b\t{-}2$};
\node at (6,7) {$\scriptscriptstyle b\t{-}2$};
\node at (8,7) {$\scriptscriptstyle b\t{-}2$};
\node at (10,7) {$\scriptscriptstyle b\t{-}2$};
\node at (12,7) {$\scriptscriptstyle b\t{-}2$};
\node at (14,7) {$\scriptscriptstyle b\t{-}2$};
\node at (16,7) {$\scriptscriptstyle b\t{-}1$};
\node at (18,7) {$\scriptscriptstyle 0$};
\node at (1,6){$\scriptscriptstyle 3b\t{-}8$};
\node at (3,6){$\scriptscriptstyle 3b\t{-}7$};
\node at (5,6){$\scriptscriptstyle 2b\t{-}4$};
\node at (7,6){$\scriptscriptstyle 2b\t{-}4$};
\node at (9,6){$\scriptscriptstyle 2b\t{-}4$};
\node at (11,6){$\scriptscriptstyle 2b\t{-}4$};
\node at (13,6){$\scriptscriptstyle 2b\t{-}4$};
\node at (15,6){$\scriptscriptstyle 2b\t{-}3$};
\node at (17,6){$\scriptscriptstyle b\t{-}1$};
\node at (19,6){$\scriptscriptstyle b\t{-}2$};
\node at (0,5) {$\scriptscriptstyle 4b\t{-}10$};
\node at (1,5) {$\scriptscriptstyle 2b\t{-}5$};
\node at (2,5) {$\scriptscriptstyle 4b\t{-}10$};
\node at (3,5) {$\scriptscriptstyle 2b\t{-}5$};
\node at (4,5) {$\scriptscriptstyle 4b\t{-}9$};
\node at (5,5) {$\scriptscriptstyle 2b\t{-}4$};
\node at (6,5) {$\scriptscriptstyle 3b\t{-}6$};
\node at (7,5) {$\scriptscriptstyle b\t{-}2$};
\node at (8,5) {$\scriptscriptstyle 3b\t{-}6$};
\node at (9,5) {$\scriptscriptstyle 2b\t{-}4$};
\node at (10,5) {$\scriptscriptstyle 3b\t{-}6$};
\node at (11,5) {$\scriptscriptstyle b\t{-}2$};
\node at (12,5) {$\scriptscriptstyle 3b\t{-}6$};
\node at (13,5) {$\scriptscriptstyle 2b\t{-}4$};
\node at (14,5) {$\scriptscriptstyle 3b\t{-}5$};
\node at (15,5) {$\scriptscriptstyle b\t{-}1$};
\node at (16,5) {$\scriptscriptstyle 2b\t{-}3$};
\node at (17,5) {$\scriptscriptstyle b\t{-}2$};
\node at (18,5) {$\scriptscriptstyle 2b\t{-}3$};
\node at (19,5) {$\scriptscriptstyle b\t{-}1$};
\node at (1,4) {$\scriptscriptstyle 3b\t{-}7$};
\node at (3,4) {$\scriptscriptstyle 3b\t{-}7$};
\node at (5,4) {$\scriptscriptstyle 3b\t{-}7$};
\node at (7,4) {$\scriptscriptstyle 3b\t{-}6$};
\node at (9,4) {$\scriptscriptstyle 2b\t{-}4$};
\node at (11,4) {$\scriptscriptstyle 3b\t{-}6$};
\node at (13,4) {$\scriptscriptstyle 2b\t{-}3$};
\node at (15,4) {$\scriptscriptstyle 2b\t{-}4$};
\node at (17,4) {$\scriptscriptstyle 2b\t{-}3$};
\node at (19,4) {$\scriptscriptstyle 2b\t{-}3$};
\node at (0,3) {$\scriptscriptstyle 3b\t{-}7$};
\node at (2,3) {$\scriptscriptstyle 2b\t{-}4$};
\node at (4,3) {$\scriptscriptstyle 2b\t{-}5$};
\node at (6,3) {$\scriptscriptstyle 3b\t{-}7$};
\node at (8,3) {$\scriptscriptstyle 2b\t{-}4$};
\node at (10,3) {$\scriptscriptstyle 2b\t{-}4$};
\node at (12,3) {$\scriptscriptstyle 2b\t{-}3$};
\node at (14,3) {$\scriptscriptstyle b\t{-}2$};
\node at (16,3) {$\scriptscriptstyle 2b\t{-}4$};
\node at (18,3) {$\scriptscriptstyle 2b\t{-}3$};
\node at (1,2) {$\scriptscriptstyle 2b\minus 4$};
\node at (3,2) {$\scriptscriptstyle b\minus 2$};
\node at (5,2) {$\scriptscriptstyle 2b\minus 5$};
\node at (7,2) {$\scriptscriptstyle 2b\t{-}5$};
\node at (9,2) {$\scriptscriptstyle 2b\t{-}4$};
\node at (11,2) {$\scriptscriptstyle b\t{-}1$};
\node at (13,2) {$\scriptscriptstyle b\t{-}2$};
\node at (15,2) {$\scriptscriptstyle b\t{-}2$};
\node at (17,2) {$\scriptscriptstyle 2b\t{-}4$};
\node at (19,2) {$\scriptscriptstyle b\t{-}1$};
\node at (0,1) {$\scriptscriptstyle b\t{-}2$};
\node at (2,1) {$\scriptscriptstyle b\t{-}2$};
\node at (4,1) {$\scriptscriptstyle b\t{-}2$};
\node at (6,1) {$\scriptscriptstyle b\t{-}3$};
\node at (8,1) {$\scriptscriptstyle 2b\t{-}5$};
\node at (10,1) {$\scriptscriptstyle b\t{-}1$};
\node at (12,1) {$\scriptscriptstyle 0$};
\node at (14,1) {$\scriptscriptstyle b\t{-}2$};
\node at (16,1) {$\scriptscriptstyle b\t{-}2$};
\node at (18,1) {$\scriptscriptstyle b\t{-}2$};
\node at (1,0) {$\scriptscriptstyle 0$};
\node at (3,0) {$\scriptscriptstyle b\t{-}2$};
\node at (5,0) {$\scriptscriptstyle 0$};
\node at (7,0) {$\scriptscriptstyle b\t{-}3$};
\node at (9,0) {$\scriptscriptstyle b\t{-}2$};
\node at (11,0) {$\scriptscriptstyle 1$};
\node at (13,0) {$\scriptscriptstyle 0$};
\node at (15,0) {$\scriptscriptstyle b\t{-}2$};
\node at (17,0) {$\scriptscriptstyle 0$};
\node at (19,0) {$\scriptscriptstyle b\t{-}2$};
\draw (1,0) ellipse (7pt and 21pt);
\draw (11,0) ellipse (7pt and 21pt);
\draw (17,0) ellipse (7pt and 21pt);
\end{tikzpicture}
\]
\[
\begin{tikzpicture}[xscale=0.5,yscale=0.2]
\node at (0,7) {$\scriptscriptstyle b\t{-}2$};
\node at (2,7) {$\scriptscriptstyle b\t{-}1$};
\node at (4,7) {$\scriptscriptstyle 0$};
\node at (6,7) {$\scriptscriptstyle b\t{-}1$};
\node at (8,7) {$\scriptscriptstyle 0$};
\node at (10,7) {$\scriptscriptstyle 0$};
\node at (12,7) {$\scriptscriptstyle b\t{-}1$};
\node at (14,7) {$\scriptscriptstyle 0$};
\node at (1,6){$\scriptscriptstyle 2b\t{-}3$};
\node at (3,6){$\scriptscriptstyle b\t{-}1$};
\node at (5,6){$\scriptscriptstyle b\t{-}1$};
\node at (7,6){$\scriptscriptstyle b\t{-}1$};
\node at (9,6){$\scriptscriptstyle 0$};
\node at (11,6){$\scriptscriptstyle b\t{-}1$};
\node at (13,6){$\scriptscriptstyle b\t{-}1$};
\node at (15,6){$\scriptscriptstyle 0$};
\node at (0,5) {$\scriptscriptstyle 2b\t{-}3$};
\node at (1,5) {$\scriptscriptstyle b\t{-}2$};
\node at (2,5) {$\scriptscriptstyle 2b\t{-}3$};
\node at (3,5) {$\scriptscriptstyle b\t{-}1$};
\node at (4,5) {$\scriptscriptstyle 2b\t{-}2$};
\node at (5,5) {$\scriptscriptstyle b\t{-}1$};
\node at (6,5) {$\scriptscriptstyle b\t{-}1$};
\node at (7,5) {$\scriptscriptstyle 0$};
\node at (8,5) {$\scriptscriptstyle b\t{-}1$};
\node at (9,5) {$\scriptscriptstyle b\t{-}1$};
\node at (10,5) {$\scriptscriptstyle b\t{-}1$};
\node at (11,5) {$\scriptscriptstyle 0$};
\node at (12,5) {$\scriptscriptstyle b\t{-}1$};
\node at (13,5) {$\scriptscriptstyle b\t{-}1$};
\node at (14,5) {$\scriptscriptstyle b\t{-}1$};
\node at (15,5) {$\scriptscriptstyle 0$};
\node at (16,5) {$\scriptscriptstyle 0$};
\node at (1,4) {$\scriptscriptstyle b\t{-}1$};
\node at (3,4) {$\scriptscriptstyle 2b\t{-}3$};
\node at (5,4) {$\scriptscriptstyle b\t{-}1$};
\node at (7,4) {$\scriptscriptstyle b\t{-}1$};
\node at (9,4) {$\scriptscriptstyle b\t{-}1$};
\node at (11,4) {$\scriptscriptstyle b\t{-}1$};
\node at (13,4) {$\scriptscriptstyle 0$};
\node at (15,4) {$\scriptscriptstyle b\t{-}1$};
\node at (17,4) {$\scriptscriptstyle 0$};
\node at (0,3) {$\scriptscriptstyle b\t{-}1$};
\node at (2,3) {$\scriptscriptstyle b\t{-}1$};
\node at (4,3) {$\scriptscriptstyle b\t{-}2$};
\node at (6,3) {$\scriptscriptstyle b\t{-}1$};
\node at (8,3) {$\scriptscriptstyle b\t{-}1$};
\node at (10,3) {$\scriptscriptstyle b\t{-}1$};
\node at (12,3) {$\scriptscriptstyle 0$};
\node at (14,3) {$\scriptscriptstyle 0$};
\node at (16,3) {$\scriptscriptstyle b\t{-}1$};
\node at (18,3) {$\scriptscriptstyle 0$};
\node at (1,2) {$\scriptscriptstyle b\t{-}1$};
\node at (3,2) {$\scriptscriptstyle 0$};
\node at (5,2) {$\scriptscriptstyle b\t{-}2$};
\node at (7,2) {$\scriptscriptstyle b\t{-}1$};
\node at (9,2) {$\scriptscriptstyle b\t{-}1$};
\node at (11,2) {$\scriptscriptstyle 0$};
\node at (13,2) {$\scriptscriptstyle 0$};
\node at (15,2) {$\scriptscriptstyle 0$};
\node at (17,2) {$\scriptscriptstyle b\t{-}1$};
\node at (19,2) {$\scriptscriptstyle 0$};
\node at (0,1) {$\scriptscriptstyle b\t{-}1$};
\node at (2,1) {$\scriptscriptstyle 0$};
\node at (4,1) {$\scriptscriptstyle 0$};
\node at (6,1) {$\scriptscriptstyle b\t{-}2$};
\node at (8,1) {$\scriptscriptstyle b\t{-}1$};
\node at (10,1) {$\scriptscriptstyle 0$};
\node at (12,1) {$\scriptscriptstyle 0$};
\node at (14,1) {$\scriptscriptstyle 0$};
\node at (16,1) {$\scriptscriptstyle 0$};
\node at (18,1) {$\scriptscriptstyle b\t{-}1$};
\node at (20,1) {$\scriptscriptstyle 0$};
\node at (1,0) {$\scriptscriptstyle 1$};
\node at (3,0) {$\scriptscriptstyle 0$};
\node at (5,0) {$\scriptscriptstyle 0$};
\node at (7,0) {$\scriptscriptstyle b\t{-}2$};
\node at (9,0) {$\scriptscriptstyle 1$};
\node at (11,0) {$\scriptscriptstyle 0$};
\node at (13,0) {$\scriptscriptstyle 0$};
\node at (15,0) {$\scriptscriptstyle 0$};
\node at (17,0) {$\scriptscriptstyle 0$};
\node at (19,0) {$\scriptscriptstyle b\t{-}1$};
\node at (21,0) {$\scriptscriptstyle 0$};
\draw (1,0) ellipse (8.4pt and 21pt);
\draw (9,0) ellipse (8.4pt and 21pt);
\draw (21,0) ellipse (8.4pt and 21pt);
\end{tikzpicture}
\]
Thus there are precisely $b-3$ maps from $M$ to $R$, one map from $M$ to $A_4$, one map from $M$ to $B_2$ and one map from $M$ to $C$.  This proves that the quiver of the reconstruction algebra is
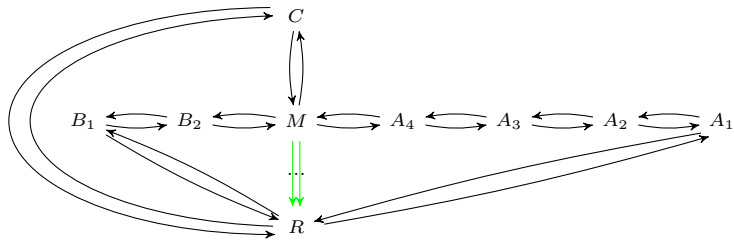
\[
\begin{tikzpicture}[xscale=1.4,yscale=1.4,bend angle=10, looseness=1]
 \node (0) at (0,0) {$\scriptstyle B_{1}$};
 \node (1) at (1,0) {$\scriptstyle B_{2}$};
 \node (1b) at (2,1) {$\scriptstyle C$};
 \node (2) at (2,0) {$\scriptstyle M$};
 \node (3) at (3,0) {$\scriptstyle A_{4}$};
 \node (4) at (4,0) {$\scriptstyle A_{3}$};
\node (5) at (5,0) {$\scriptstyle A_{2}$};
\node (6) at (6,0) {$\scriptstyle A_{1}$};
\node (R) at (2,-1) {$\scriptstyle R$};
\draw [bend right,->] (0) to (1);
\draw [bend right,->] (1) to (0);
\draw [bend right,->] (1) to (2);
\draw [bend right,->] (2) to (1);
\draw [bend right,->] (2) to (3);
\draw [bend right,->] (3) to (2);
\draw [bend right,->] (3) to (4);
\draw [bend right,->] (4) to (3);
\draw [bend right,->] (4) to (5);
\draw [bend right,->] (5) to (4);
\draw [bend right,->] (5) to (6);
\draw [bend right,->] (6) to (5);
\draw [bend right,->] (2) to (1b);
\draw [bend right,->] (1b) to (2);
\draw [bend right=5,->] (0) to (R);
\draw [bend right=5,->] (R) to (0);
\draw [bend right=3,->] ($(6)+(-150:6.5pt)$) to (R);
\draw [bend right=3,->] ($(R)+(5:7.5pt)$) to ($(6)+(-130:5.5pt)$);
\draw[->] (-0.5,0) arc (180:95:2.5cm and 1cm); 
\draw (-0.5,0) arc (-180:-95:2.5cm and 1cm);  
\draw (-0.7,0) arc (180:91:2.5cm and 1.075cm); 
\draw[->] (-0.7,0) arc (-180:-90:2.5cm and 1.075cm);  
\draw[->,green] ($(2)+(-100:5.5pt)$) -- ($(R) + (100:5.5pt)$);
\draw[->,green] ($(2)+(-80:5.5pt)$) -- ($(R) + (80:5.5pt)$);
 \node at (2,-0.5) {$\scriptstyle ...$};
\end{tikzpicture}
\]
where there are $b-3$ arrows from $M$ to $R$.
\end{example}

\begin{example}
Consider the group $\mathbb{I}_{7}$.  By \cite[9.3]{Iyama_Wemyss_specials}  the indecomposable special CM modules occupy the following positions in the AR quiver:

{\tiny
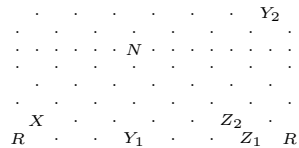
{
\[
\xymatrix@C=-1pt@R=-0pt{
&.&&.&&.&&.&&.&&.&&{}\drop{Y_2}{}& \\
.&&.&&.&&.&&.&&.&&.&&.&\\
.&.&.&.&.&.&{}\drop{N}{}&.&.&.&.&.&.&.&.\\
.&&.&&.&&.&&.&&.&&.&&.& \\
&.&&.&&.&&.&&.&&.&&.&\\
.&&.&&.&&.&&.&&.&&.&&.\\
&{}\drop{X}{}&&.&&.&&.&&.&&{}\drop{Z_2}{}&&.&  \\
{}\drop{R}{}&&.&&.&&{}\drop{Y_1}{}&&.&&.&&{}\drop{Z_1}{}&&{}\drop{R}{}}
\]
}}Now it is clear that the only map out of $R$ is to $X$.  Further

{\tiny{
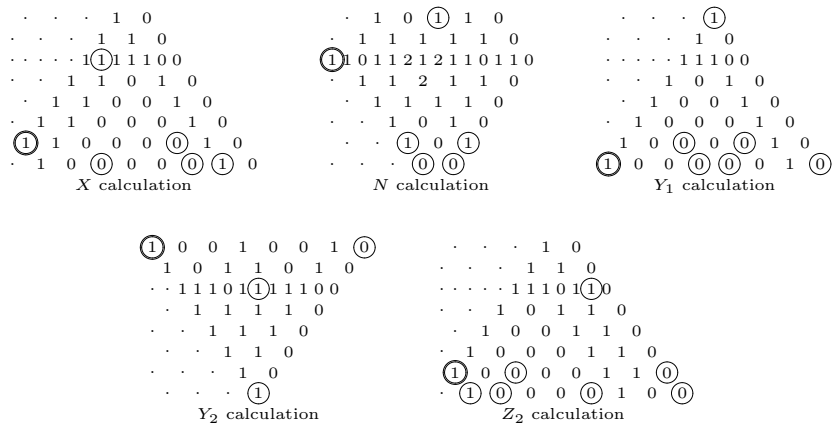
\[
\begin{array}{ccc}
\begin{array}{c}
\xymatrix@C=-4pt@R=-2pt{
&.&&.&&.&&1&&0&\\
.&&.&&.&&1&&1&&0\\
.&.&.&.&.&1&{{}\drop\xycircle<4pt,4pt>{}1}&1&1&1&0&0\\
.&&.&&1&&1&&0&&1&&0\\
&.&&1&&1&&0&&0&&1&&0 \\
.&&1&&1&&0&&0&&0&&1&&0 \\
&{{}\drop\xycircle<4pt,4pt>{}{}\drop\xycircle<4.5pt,4.5pt>{}1}&&1&&0&&0&&0&&{{}\drop\xycircle<4pt,4pt>{}0}&&1&&0\\
.&&1&&0&&{{}\drop\xycircle<4pt,4pt>{}0}&&0&&0&&{{}\drop\xycircle<4pt,4pt>{}0}&&{{}\drop\xycircle<4pt,4pt>{}1}&&0
}\\
X \mbox{ calculation}
\end{array}&
\begin{array}{c}
\xymatrix@C=-4pt@R=-2pt{
&.&&1&&0&&{{}\drop\xycircle<4pt,4pt>{}1}&&1&&0\\
.&&1&&1&&1&&1&&1&&0\\
{{}\drop\xycircle<4pt,4pt>{}{}\drop\xycircle<4.5pt,4.5pt>{}1}&1&0&1&1&2&1&2&1&1&0&1&1&0\\
.&&1&&1&&2&&1&&1&&0\\
&.&&1&&1&&1&&1&&0 \\
.&&.&&1&&0&&1&&0 \\
&.&&.&&{{}\drop\xycircle<4pt,4pt>{}1}&&0&&{{}\drop\xycircle<4pt,4pt>{}1}\\
.&&.&&.&&{{}\drop\xycircle<4pt,4pt>{}0}&&{{}\drop\xycircle<4pt,4pt>{}0}\\
}\\
N \mbox{ calculation}
\end{array}&
\begin{array}{c}
\xymatrix@C=-4pt@R=-2pt{
&.&&.&&.&&{{}\drop\xycircle<4pt,4pt>{}1}&\\
.&&.&&.&&1&&0\\
.&.&.&.&.&1&1&1&0&0\\
.&&.&&1&&0&&1&&0\\
&.&&1&&0&&0&&1&&0 \\
.&&1&&0&&0&&0&&1&&0 \\
&1&&0&&{{}\drop\xycircle<4pt,4pt>{}0}&&0&&{{}\drop\xycircle<4pt,4pt>{}0}&&1&&0\\
{{}\drop\xycircle<4pt,4pt>{}{}\drop\xycircle<4.5pt,4.5pt>{}1}&&0&&0&&{{}\drop\xycircle<4pt,4pt>{}0}&&{{}\drop\xycircle<4pt,4pt>{}0}&&0&&1&&{{}\drop\xycircle<4pt,4pt>{}0}
}\\
Y_1 \mbox{ calculation}
\end{array}
\end{array}
\]}}
{\tiny{
\[
\begin{array}{cc}
\begin{array}{c}
\xymatrix@C=-4pt@R=-2pt{
{{}\drop\xycircle<4pt,4pt>{}{}\drop\xycircle<4.5pt,4.5pt>{}1}&&0&&0&&1&&0&&0&&1&&{{}\drop\xycircle<4pt,4pt>{}0}\\
&1&&0&&1&&1&&0&&1&&0\\
.&.&1&1&1&0&1&{{}\drop\xycircle<4pt,4pt>{}1}&1&1&1&0&0\\
&.&&1&&1&&1&&1&&0\\
.&&.&&1&&1&&1&&0 \\
&.&&.&&1&&1&&0 \\
.&&.&&.&&1&&0\\
&.&&.&&.&&{{}\drop\xycircle<4pt,4pt>{}1}\\
}\\
Y_2 \mbox{ calculation}
\end{array}&
\begin{array}{c}
\xymatrix@C=-4pt@R=-2pt{
&.&&.&&.&&1&&0\\
.&&.&&.&&1&&1&&0\\
.&.&.&.&.&1&1&1&0&1&{{}\drop\xycircle<4pt,4pt>{}1}&0\\
.&&.&&1&&0&&1&&1&&0\\
&.&&1&&0&&0&&1&&1&&0 \\
.&&1&&0&&0&&0&&1&&1&&0 \\
&{{}\drop\xycircle<4pt,4pt>{}{}\drop\xycircle<4.5pt,4.5pt>{}1}&&0&&{{}\drop\xycircle<4pt,4pt>{}0}&&0&&0&&1&&1&&{{}\drop\xycircle<4pt,4pt>{}0}\\
.&&{{}\drop\xycircle<4pt,4pt>{}1}&&{{}\drop\xycircle<4pt,4pt>{}0}&&0&&0&&{{}\drop\xycircle<4pt,4pt>{}0}&&1&&0&&{{}\drop\xycircle<4pt,4pt>{}0}
}\\
Z_2 \mbox{ calculation}
\end{array}
\end{array}
\]
}}Finally the number of irreducible maps out of $Z_1$ is determined by

{\tiny{
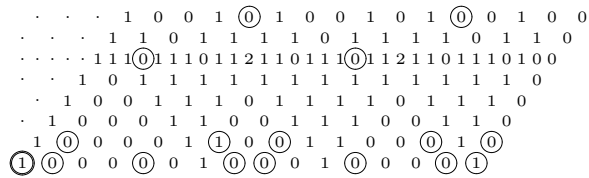
\[
\begin{array}{c}
\xymatrix@C=-4pt@R=-2pt{
&.&&.&&.&&1&&0&&0&&1&&{{}\drop\xycircle<4pt,4pt>{}0}&&1&&0&&0&&1&&0&&1&&{{}\drop\xycircle<4pt,4pt>{}0}&&0&&1&&0&&0\\
.&&.&&.&&1&&1&&0&&1&&1&&1&&1&&0&&1&&1&&1&&1&&0&&1&&1&&0\\
.&.&.&.&.&1&1&1&{{}\drop\xycircle<4pt,4pt>{}0}&1&1&1&0&1&1&2&1&1&0&1&1&1&{{}\drop\xycircle<4pt,4pt>{}0}&1&1&2&1&1&0&1&1&1&0&1&0&0\\
.&&.&&1&&0&&1&&1&&1&&1&&1&&1&&1&&1&&1&&1&&1&&1&&1&&0\\
&.&&1&&0&&0&&1&&1&&1&&0&&1&&1&&1&&1&&0&&1&&1&&1&&0\\
.&&1&&0&&0&&0&&1&&1&&0&&0&&1&&1&&1&&0&&0&&1&&1&&0\\
&1&&{{}\drop\xycircle<4pt,4pt>{}0}&&0&&0&&0&&1&&{{}\drop\xycircle<4pt,4pt>{}1}&&0&&{{}\drop\xycircle<4pt,4pt>{}0}&&1&&1&&0&&0&&{{}\drop\xycircle<4pt,4pt>{}0}&&1&&{{}\drop\xycircle<4pt,4pt>{}0}\\
{{}\drop\xycircle<4pt,4pt>{}{}\drop\xycircle<4.5pt,4.5pt>{}1}&&{{}\drop\xycircle<4pt,4pt>{}0}&&0&&0&&{{}\drop\xycircle<4pt,4pt>{}0}&&0&&1&&{{}\drop\xycircle<4pt,4pt>{}0}&&{{}\drop\xycircle<4pt,4pt>{}0}&&0&&1&&{{}\drop\xycircle<4pt,4pt>{}0}&&0&&0&&{{}\drop\xycircle<4pt,4pt>{}0}&&{{}\drop\xycircle<4pt,4pt>{}1}
}
\end{array}
\]
}}and so the quiver is
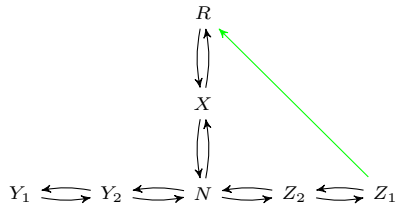
\[
\begin{tikzpicture}[xscale=1.2,yscale=1.2,bend angle=10, looseness=1]
 \node (0) at (0,0) {$\scriptstyle Y_{1}$};
 \node (1) at (1,0) {$\scriptstyle Y_{2}$};
 \node (1b) at (2,1) {$\scriptstyle X$};
 \node (2) at (2,0) {$\scriptstyle N$};
 \node (3) at (3,0) {$\scriptstyle Z_{2}$};
 \node (4) at (4,0) {$\scriptstyle Z_{1}$};
\node (R) at (2,2) {$\scriptstyle R$};
\draw [bend right,->] (0) to (1);
\draw [bend right,->] (1) to (0);
\draw [bend right,->] (1) to (2);
\draw [bend right,->] (2) to (1);
\draw [bend right,->] (2) to (3);
\draw [bend right,->] (3) to (2);
\draw [bend right,->] (3) to (4);
\draw [bend right,->] (4) to (3);
\draw [bend right,->] (2) to (1b);
\draw [bend right,->] (1b) to (2);
\draw [bend right,->] (1b) to (R);
\draw [bend right,->] (R) to (1b);
\draw[->,green] (4) -- (R);
\end{tikzpicture}
\]
\end{example}

\section{Quotient Singularities}\label{groups}
We now illustrate the details of the previous sections for all quotient singularities.  The finite small subgroups of $\GL(2,\C{})$ and their dual graphs were classified by Brieskorn \cite{Brieskorn}, but here we use the notation from Riemenschneider \cite{Riemenschneider_invarianten}.  The
classification of finite small subgroups of $\GL(2,\C{})$ is summarized in the following table:
\[
\begin{array}{ccc}
\t{Type} & \t{Notation} & \t{Conditions}\\
\mathbb{A} & \frac{1}{r}(1,a)=\left\langle \left(\begin{smallmatrix} \ve_r & 0\\ 0& \ve_r^a
\end{smallmatrix}\right)\right\rangle & 1<a<r, (r,a)=1 \\
\mathbb{D}& \begin{array}{cc}
\mathbb{D}_{n,q}=\left\{ \begin{array}{cc} \langle \psi_{2q},
\tau, \varphi_{2(n-q)} \rangle& \mbox{if } n-q\equiv 1 \mbox{ mod
}2\\\langle \psi_{2q}, \tau\varphi_{4(n-q)} \rangle& \mbox{if }
n-q\equiv 0 \mbox{ mod }2
\end{array}\right.
\end{array}& 1<q<n, (n,q)=1\\
\mathbb{T}
&\mathbb{T}_m=\left\{ \begin{array}{cc} \langle \psi_{4},
\tau, \eta, \varphi_{2m} \rangle& \mbox{if } m\equiv 1,5 \mbox{ mod
}6\\\langle \psi_{4},\tau,\eta\varphi_{6m} \rangle& \mbox{if }
m\equiv 3 \mbox{ mod }6
\end{array}\right. & m\equiv 1,3,5 \mbox{ mod }6\\
\mathbb{O}&
\mathbb{O}_m=\langle \psi_{8},\tau,\eta,
\varphi_{2m} \rangle
& m\equiv 1,5,7,11 \mbox{ mod }12
\\
\mathbb{I}&
\mathbb{I}_m=\langle \left(\begin{smallmatrix}0&-1\\1&0\end{smallmatrix}\right),\omega,\iota, \varphi_{2m} \rangle
&  \begin{array}{rl}m\equiv& 1, 7,11,13,17,19,\\ &23,29 \mbox{ mod }30\end{array}
\end{array}
\]
with the matrices
\[
\begin{array}{c}
\begin{array}{cccc}
\psi_k= \begin{pmatrix}\e_k & 0\\ 0& \e_k^{-1}
\end{pmatrix} &\tau = \begin{pmatrix}0 & \e_4\\ \e_4& 0
\end{pmatrix}&\varphi_k= \begin{pmatrix}\e_k & 0\\ 0& \e_k
\end{pmatrix}&\eta=\frac{1}{\sqrt{2}} \begin{pmatrix}\e_8 & e_8^3\\ \e_8& \e_8^7 \end{pmatrix}
\end{array}\\
\begin{array}{cc}
\omega= \begin{pmatrix}\e_5^3 & 0\\ 0& \e_5^2
\end{pmatrix}&
\iota=\frac{1}{\sqrt{5}} \begin{pmatrix}\e_5^4-\e_5 & \e_5^2-\e_5^3\\ \e_5^2-\e_5^3& \e_5-\e_5^4
\end{pmatrix}
\end{array}
\end{array}
\]
where $\e_t$ is a primitive ${t}^{\rm th}$ root of unity. 

\begin{definition}
Given a labelled Dynkin diagram we associate the corresponding $Z_f$. We call $Z_f$ \\
\t{(i)} maximal if it takes the same value as the $Z_f$ associated to the same Dynkin diagram labelled with $(-2)$-curves.\\
\t{(ii)} reduced if $Z_f$ consists only of ones.
\end{definition}

Now to build the quiver of the reconstruction algebra associated to $\C{2}/G$, using Theorem~\ref{main} and Lemma~\ref{basecase} it is easy to transfer the  intersection theory combinatorics into the following three rules:
\begin{itemize}
\item[(1)] \textbf{If $Z_f$ is maximal} then connect $\star$ to make the extended Dynkin diagram, double the quiver and for every $\alpha_i>2$, add an extra $\alpha_i-2$ arrows from that vertex to $\star$.
\item[(2)] \textbf{If $Z_f$ is reduced \emph{but not maximal}} then we are not in type $\mathbb{A}$ so the Dynkin diagram has a vertex which connects to three others; we call this vertex the middle vertex.  To build the quiver, connect $\star$ to the vertices at the end of each `arm' coming out of the middle vertex, double the resulting quiver and then add extra arrows subject to the following rule: if some $\alpha_i>2$, add an extra $\alpha_i-2$ arrows from that vertex to $\star$, \emph{except} at the middle vertex of the Dynkin diagram where we only add an extra $\alpha_t-3$ arrows. 
\item[(3)] \textbf{If $Z_f$ is neither maximal nor reduced.}  Denote by $C$ the vertex with $\alpha_C\geq 3$ which is closest to the middle vertex.  Now inside $Z_f$ we may find the largest subvector which is the maximal $Z_f$ of some Dynkin subdiagram $Q^\prime$ of type $D$.  Considering this subdiagram $Q^\prime$, connect $\star$ to form the extended Dynkin diagram of $Q^{\prime}$.  Then also connect $\star$ to those vertices at the end of each arm of the original $Q$ which do not lie in $Q^\prime$, double the resulting quiver and add extra arrows subject to the following rule:  if some $\alpha_i>2$, add an extra $\alpha_i-2$ arrows from that vertex to $\star$, \emph{except} at the vertex $C$, where we only add $\alpha_C-3$ extra arrows. 
\end{itemize}
\begin{remark}
Case (3) is somewhat artificial since it does not make sense for an arbitrary labelled Dynkin diagram.  However for all the labelled Dynkin diagrams coming from quotient singularities (some in type $\mathbb{D}$, then  $\mathbb{I}_{13}$ and $\mathbb{I}_{23}$) it gives the correct quiver.  See Example~\ref{type3example} below.
\end{remark}
\begin{remark}
We remark that there are both geometric and algebraic consequences of the above trichotomy.  It turns out that geometrically in (1) the minimal resolution is very similar to the minimal resolutions arising from ADE quotients, where in (2) the minimal resolution is almost identical to those found in the toric case.  Case (3) is somewhere in between these two extremes.
\end{remark}

Case (3) is illustrated in Example~\ref{type3example} below --- \ref{type3example}(i) is an example from type $\mathbb{D}$ whereas \ref{type3example}(ii) is $\mathbb{I}_{13}$.  We label by dotted lines the largest subvector which is the maximal $Z_f$ of some Dynkin subdiagram of type $D$.
\begin{example}\label{type3example}
\[
\begin{tikzpicture}[yscale=1.2]
\node at (-3,0) {(i)};
\node at (0,0) {\begin{tikzpicture}[xscale=0.85,yscale=1]
 \node (0) at (0,0) [vertex] {};
 \node (1) at (1,0) [vertex] {};
 \node (1b) at (1,0.75) [vertex] {};
 \node (2) at (2,0) [vertex] {};
 \node (3) at (3,0) [vertex] {};
 \node (4) at (4,0) [vertex] {};
 \node (5) at (5,0)[vertex] {};
 \node (0a) at (-0.1,-0.25) {$\scriptstyle - 2$};
 \node (1a) at (0.9,-0.25) {$\scriptstyle -2$};
 \node (1ba) at (0.65,0.75) {$\scriptstyle - 2$};
 \node (2a) at (1.9,-0.25) {$\scriptstyle - 2$};
 \node (2a) at (2.9,-0.25) {$\scriptstyle - 5$};
 \node (4a) at (3.9,-0.25) {$\scriptstyle -2$};
 \node (5a) at (4.9,-0.25) {$\scriptstyle - 3$};
\draw [-] (0) -- (1);
\draw [-] (1) -- (2);
\draw [-] (2) -- (3);
\draw [-] (3) -- (4);
\draw [-] (4) -- (5);
\draw [-] (1) -- (1b);
\end{tikzpicture}};
\node at (3.5,0) {$\scriptstyle Z_{f}=$};
\node at (6,0.1)
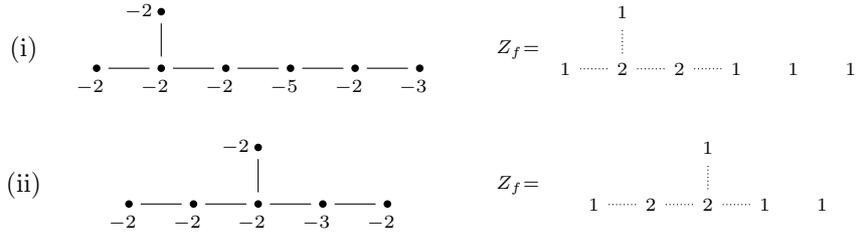
{\begin{tikzpicture}[xscale=0.75,yscale=1]
 \node (0) at (0,0)  {$\scriptstyle 1$};
 \node (1) at (1,0)  {$\scriptstyle 2$};
 \node (1b) at (1,0.75)  {$\scriptstyle 1$};
 \node (2) at (2,0)  {$\scriptstyle 2$};
 \node (3) at (3,0) {$\scriptstyle 1$};
 \node (4) at (4,0)  {$\scriptstyle 1$};
 \node (5) at (5,0) {$\scriptstyle 1$};
\draw [densely dotted] (0) -- (1);
\draw [densely dotted] (1) -- (2);
\draw [densely dotted] (2) -- (3);
\draw [densely dotted] (1) -- (1b);
\end{tikzpicture}};
\node at (-3,-1.5) {(ii)};
\node at (0,-1.5) {\begin{tikzpicture}[xscale=0.85,yscale=1]
 \node (0) at (0,0) [vertex] {};
 \node (1) at (1,0) [vertex] {};
 \node (1b) at (2,0.75) [vertex] {};
 \node (2) at (2,0) [vertex] {};
 \node (3) at (3,0) [vertex] {};
 \node (4) at (4,0) [vertex] {};
 \node (0a) at (-0.1,-0.25) {$\scriptstyle - 2$};
 \node (1a) at (0.9,-0.25) {$\scriptstyle -2$};
 \node (1ba) at (1.65,0.75) {$\scriptstyle - 2$};
 \node (2a) at (1.9,-0.25) {$\scriptstyle - 2$};
 \node (2a) at (2.9,-0.25) {$\scriptstyle - 3$};
 \node (4a) at (3.9,-0.25) {$\scriptstyle -2$};
\draw [-] (0) -- (1);
\draw [-] (1) -- (2);
\draw [-] (2) -- (3);
\draw [-] (3) -- (4);
\draw [-] (2) -- (1b);
\end{tikzpicture}};
\node at (3.5,-1.5) {$\scriptstyle Z_{f}=$};
\node at (6,-1.4)
{\begin{tikzpicture}[xscale=0.75,yscale=1]
 \node (0) at (0,0)  {$\scriptstyle 1$};
 \node (1) at (1,0)  {$\scriptstyle 2$};
 \node (1b) at (2,0.75)  {$\scriptstyle 1$};
 \node (2) at (2,0)  {$\scriptstyle 2$};
 \node (3) at (3,0) {$\scriptstyle 1$};
 \node (4) at (4,0)  {$\scriptstyle 1$};
\draw [densely dotted] (0) -- (1);
\draw [densely dotted] (1) -- (2);
\draw [densely dotted] (2) -- (3);
\draw [densely dotted] (2) -- (1b);
\end{tikzpicture}};
\end{tikzpicture}
\]
\end{example}

In the sections below, when stating the dual graph we are referring to
Brieskorn \cite[2.11]{Brieskorn}.  On the vertices of the quivers, for the convenience of the reader we write $Z_{f}$ and hence the dimension of the irreducible special representation that corresponds to that vertex.
In types $\mathbb{T}$,  $\mathbb{O}$ and  $\mathbb{I}$, labelled
arrows with $...$ means that there are $b-3$ arrows from that vertex
to $\star$, where if $b=3$ there are no arrows.

\section{Type $\mathbb{A}$}\label{Section:TypeA}
Given $G=\frac{1}{r}(1,a)$ consider the Jung--Hirzebruch continued fraction expansion $\frac{r}{a}=[\alpha_1,\hdots,\alpha_n]$.  It was first proved in \cite{Wemyss_reconstruct_A} that the quiver of the reconstruction algebra can be described as follows:  take the labelled dual graph of the minimal resolution
\[

\]
and associate the double quiver of the extended
Dynkin diagram:
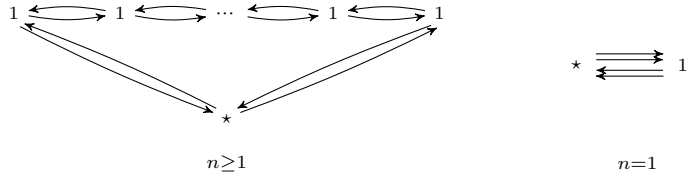
\[
%
\]
Then for every $\alpha_i>2$, add an extra $\alpha_i-2$ arrows from vertex $i$ to
$\star$.
\begin{example}
Consider $\frac{1}{5}(1,3)$ and $\frac{1}{40}(1,11)$.  Then $\frac{5}{3}=[2,3]$ and $\frac{40}{11}=[4,3,4]$ so the quivers and dual graphs are
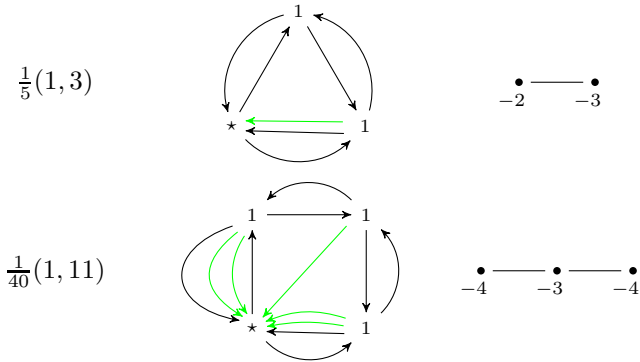
\[
%
 
\]
\end{example}

\section{Type $\mathbb{D}$}\label{Section:TypeD}
Given $\mathbb{D}_{n,q}$, the dual graph of the minimal resolution
of $\C{2}/\mathbb{D}_{n,q}$ is
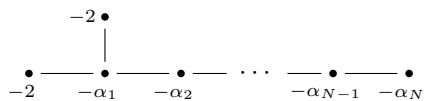
\[
%
\]
where the $\alpha$'s come from the Jung--Hirzebruch continued
fraction expansion of $\frac{n}{q}$
\[
\frac{n}{q}=[\alpha_1,\hdots,\alpha_{N} ].
\]
\noindent
\textbf{Case (i): $Z_f$ is reduced.}
In this case to the labelled Dynkin diagram above attach $\star$ and double the quiver to produce
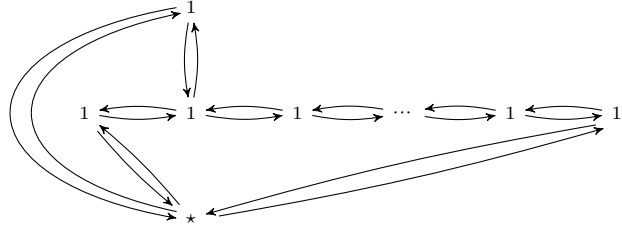
\[
%
\]
If $\sum_{i=1}^{N}(\alpha_i-2)\geq 2$ then add extra arrows as
follows:
\begin{itemize}
\item If $\alpha_1>3$ then add an extra $\alpha_1-3$ arrows from that vertex to $\star$.
\item If $\alpha_i>2$ with $i\geq 2$ then add an extra $\alpha_i-2$ arrows from that vertex to $\star$.
\end{itemize}

\begin{example}
Consider $\mathbb{D}_{52,11}$ then $\frac{52}{11}=[5,4,3]$ and so
the quiver and dual graph is
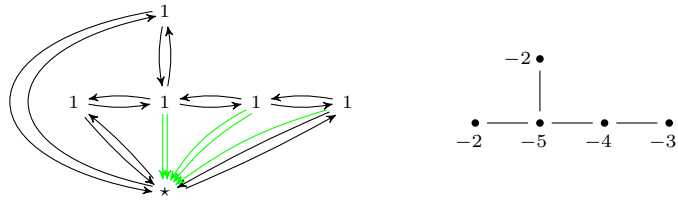
\[
%
\]
\end{example}

\noindent
\textbf{Case(ii): $Z_f$ not reduced or maximal.}
Let $\nu$ be the number of two-dimensional special representations, then to the above labelled Dynkin diagram attach $\star$ to the $\nu^{\rm th}$ and $N^{\rm th}$ vertices and double the quiver:
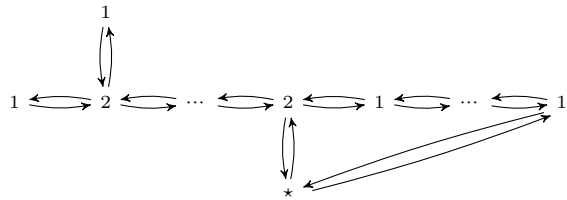
\[
%
\]
If
$\sum_{i=1}^{N}(\alpha_i-2)=\sum_{i=\nu+1}^{N}(\alpha_i-2)\geq 2$
then add extra arrows as follows:
\begin{itemize}
\item If $\alpha_{\nu+1}>3$ then add an extra $\alpha_{\nu+1}-3$ arrows from that vertex to $\star$
\item If some other $\alpha_i>2$ then add an extra $\alpha_i-2$ arrows from that vertex to $\star$.
\end{itemize}

\begin{example}
Consider $\mathbb{D}_{56,39}$ then $\frac{56}{39}=[2,2,5,2,3]$ so the quiver and dual graph is
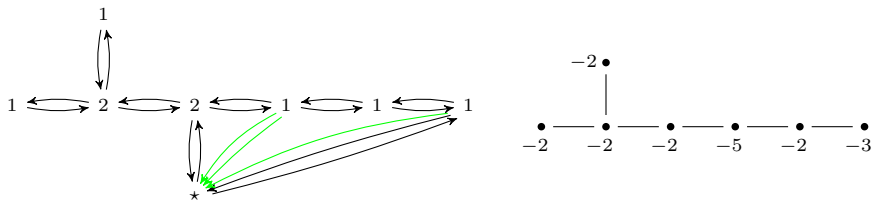
\[
%
\]
\end{example}

\noindent
\textbf{Case (iii): $Z_f$ maximal.}
To the above labelled Dynkin diagram first associate the quiver of the preprojective algebra of the extended Dynkin diagram
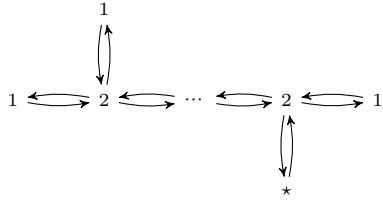
\[
%
\]
If $\alpha_N>2$ then add an extra
$\alpha_N-2$ arrows from vertex $N$ to $\star$.
\begin{example}
For the groups $\mathbb{D}_{7,4}$ and $\mathbb{D}_{7,5}$, the quivers and dual graphs are
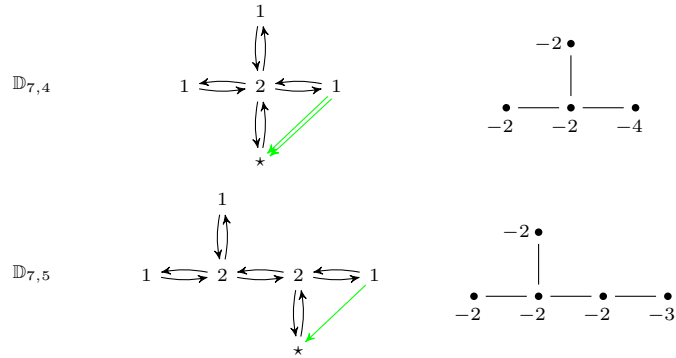
\[
%
\]
\end{example}

\section{Type $\mathbb{T}$}\label{Section:TypeT}
These are the groups $\mathbb{T}_m$ where $m\equiv 1,3$ or $5$ mod $6$.\\

\noindent
\textbf{The case $m\equiv 1$.}
In this subfamily $m=6(b-2)+1$.
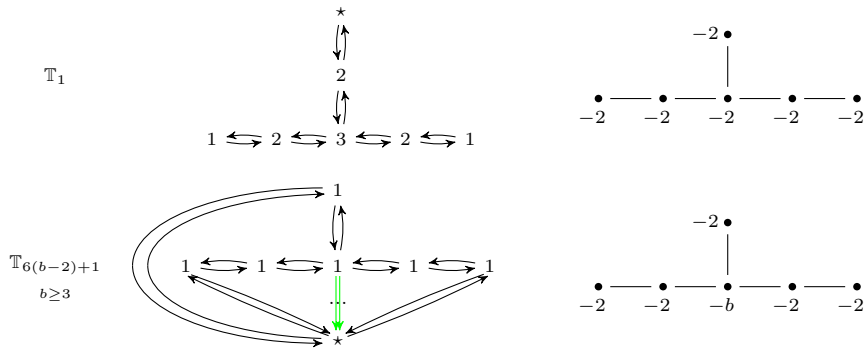
\[
%
\]

\noindent
\textbf{The case $m\equiv 3$.} In this subfamily $m=6(b-2)+3$.
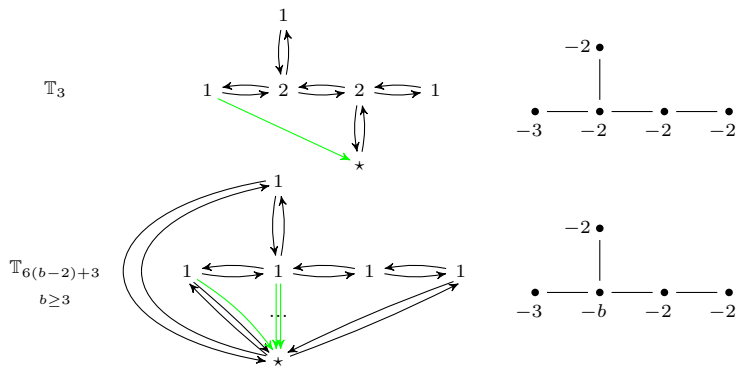
\[
%
\]

\noindent
\textbf{The case $m\equiv 5$}.  In this subfamily $m=6(b-2)+5$.
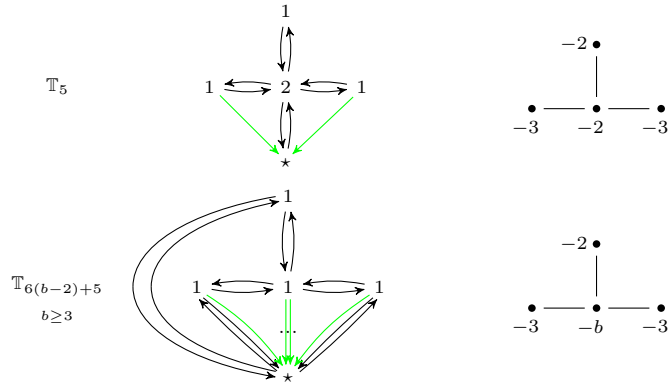
\[
%
\]

\section{Type $\mathbb{O}$}\label{Section:TypeO}
These are the groups $\mathbb{O}_m$ where $m\equiv 1,5,7$ or $11$ mod $12$.  \\

\noindent
\textbf{The case $m\equiv 1$.} In this subfamily $m=12(b-2)+1$.
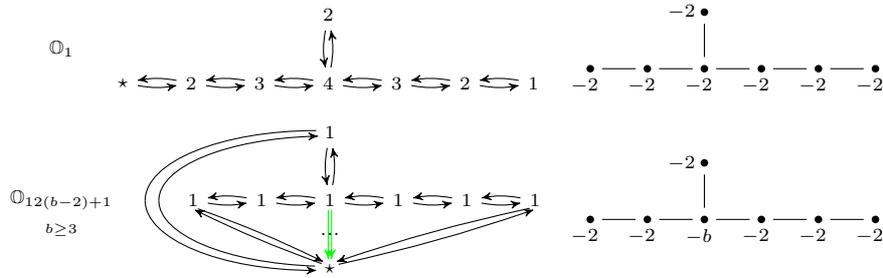
\[
%
\]

\noindent
\textbf{The case $m\equiv 5$.} In this subfamily $m=12(b-2)+5$.
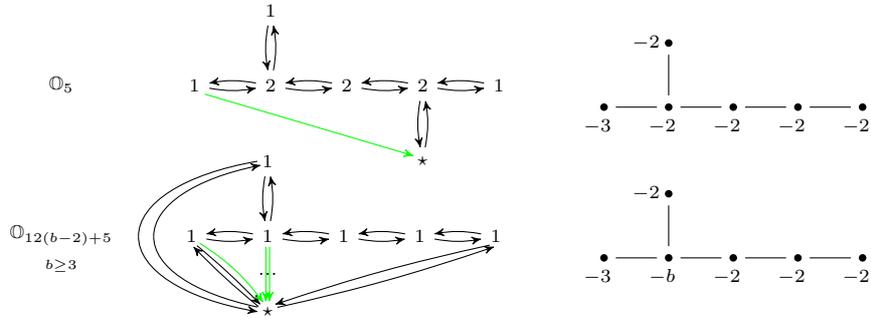
\[
%
\]

\noindent
\textbf{The case $m\equiv 7$.} In this subfamily $m=12(b-2)+7$.
\[
%
\]

\section{Type $\mathbb{I}$}\label{Section:TypeI}
These are the groups $\mathbb{I}_m$ where $m\equiv 1,7, 11, 13, 17, 19, 23$ or $29$ mod $30$.\\

\noindent
\textbf{The case $m\equiv 1$.}  In this subfamily $m=30(b-2)+1$.
\[
%
};
\end{tikzpicture}
\]

\noindent
{\bf Acknowledgement.}
This paper was originally written whilst the author visited Nagoya University
with the Cecil King Travel Scholarship from the London Mathematical
Society, and was substantially revised when the author held a JSPS Postdoctoral Fellowship, also at Nagoya University.   The author would like to thank the Cecil King Foundation, the LMS and the JSPS for funding this work, and Nagoya University for kind hospitality.  The author would also like to thank Osamu Iyama and Akira Ishii for helpful conversations and comments.


\begin{thebibliography}{10}
\bibitem[1]{Art66}
M. Artin,
\emph{On isolated rational singularities of surfaces.}  Amer. J. Math.  \textbf{88} (1966) 129--136.

\bibitem[2]{Auslander_AR=McKay}
M.~Auslander \emph{Rational singularities and almost split sequences}, Trans. Amer. Math. Soc. \textbf{293} (1986), no.~2, 511--531.

\bibitem[3]{AR_McKayGraphs}
M.~Auslander and I.~Reiten, \emph{{M}c{K}ay quivers and extended
{D}ynkin  diagrams}, Trans. Amer. Math. Soc. \textbf{293} (1986), no.~1, 293--301.

\bibitem[4]{Bridgeland}
T.~Bridgeland, \emph{Flops and derived categories}. Invent. Math. \textbf{147} (2002), no.~3, 613--632

\bibitem[5]{Brieskorn}
E.~Brieskorn, \emph{Rationale singularit{\"{a}}ten komplexer fl{\"{a}}chen},
  Invent. Math. \textbf{4} (1968), 336--358.

\bibitem[6]{BIKS}
A. Buan, O. Iyama, I. Reiten and D. Smith, \emph{Mutation of cluster-tilting objects and potentials},
arXiv:0804.3813 (2008)

\bibitem[7]{Esnault}
H.~Esnault, \emph{Reflexive modules on quotient surface singularities}, J. Reine Angew. Math. \textbf{362} (1985), 63--71.

\bibitem[8]{Ishii} 
A.~Ishii, \emph{On the McKay correspondence for a finite small subgroup of $\GL(2,\mathbb{C})$}, J. Reine Angew. Math. \textbf{549} (2002), 221--233

\bibitem[9]{Iyama_tau1} 
O. Iyama, \emph{$\tau$-categories I: Ladders},
Algebr. Represent. Theory \textbf{8} (2005), no. 3, 297--321.

\bibitem[10]{Iyama_Wemyss_specials}
O. Iyama and M. Wemyss, \emph{The classification of special Cohen Macaulay modules}, Math. Z. \textbf{265} (2010), no.1, 41--83.

\bibitem[11]{McKay_original}
J.~McKay, \emph{Graphs, singularities, and finite groups}, Proc.
Sympos. Pure
  Math. \textbf{37} (1980), 183--186.

\bibitem[12]{Riemenschneider_invarianten}
O.~Riemenschneider, \emph{Invarianten endlicher {U}ntergruppen},
Math. Z.  \textbf{153} (1977), 37--50.
  
\bibitem[13]{TT}
L. D. Tr{\'a}ng and M. Tosun, \emph{Combinatorics of rational singularities}, Comment. Math. Helv. \textbf{79} (2004), no. 3, 582--604.

\bibitem[14]{VdB}
M.~Van den Bergh \emph{Three-dimensional flops and noncommutative rings}, 
Duke Math. J. \textbf{122} (2004), no.~3, 423--455. 

\bibitem[15]{Wemyss_reconstruct_A}
M.~Wemyss, \emph{Reconstruction algebras of type ${A}$}, to appear Trans. AMS.

\bibitem[16]{Wemyss_reconstruct_D(i)}
M.~Wemyss, \emph{Reconstruction algebras of type ${D}$} (I), arXiv:0905.1154 (2009).

\bibitem[17]{Wemyss_reconstruct_D(ii)}
M.~Wemyss, \emph{Reconstruction algebras of type ${D}$} (II), arXiv:0905.1155 (2009).

\bibitem[18]{Wunram_generalpaper}
J.~Wunram, \emph{Reflexive modules on quotient surface singularities},
  Mathematische Annalen \textbf{279} (1988), no.~4, 583--598.

\end{thebibliography}
\end{document}